\documentclass[11pt]{article}

\usepackage{fullpage}  

\usepackage[utf8]{inputenc}

\usepackage{amsmath}  
\usepackage{amsthm}  
\usepackage{amssymb}  
\usepackage{bm}  
\usepackage{mathtools}  

\usepackage{enumerate}  

\usepackage{graphicx}  
\usepackage[dvipsnames]{xcolor}  
\usepackage[small,bf]{caption}  

\graphicspath{{figures/}}  
\graphicspath{{../figures/}}  

\usepackage{breakcites}  
\usepackage{url}  
\usepackage{hyperref}  

\hypersetup{colorlinks, citecolor=black, filecolor=black,
    		linkcolor=black, urlcolor=black}

\usepackage{algorithm}
\usepackage[noend]{algpseudocode}

\usepackage{comment}  

\usepackage{titlesec}
\titleformat*{\section}{\Large\centering\normalfont\scshape}
\titleformat{\subsection}[runin]{\large\raggedright\itshape}{\thesubsection}{1em}{}
\usepackage[title]{appendix}

\theoremstyle{plain}  
	\newtheorem{theorem}{Theorem}
	\newtheorem{proposition}[theorem]{Proposition}
	\newtheorem{lemma}[theorem]{Lemma}
	\newtheorem{corollary}[theorem]{Corollary}
	
\theoremstyle{definition}

	\newtheorem{definition}[theorem]{Definition}
	\newtheorem{exampleEnv}[theorem]{Example}
			{\pushQED{\qed}\exampleEnv}
			{\popQED\endexampleEnv}
		\newenvironment{example*}{\exampleEnv}{\endexampleEnv}
	\newtheorem{remarkEnv}[theorem]{Remark}
		\newenvironment{remark}  
			{\pushQED{\qed}\remarkEnv}
			{\popQED\endremarkEnv}
		\newenvironment{remark*}{\remarkEnv}{\endremarkEnv}
	
	\newtheorem*{solutionEnv}{Solution}
		\ifdefined\showsolution
			\newenvironment{solution}
				{\pushQED{\qed}\solutionEnv}
				{\popQED\endsolutionEnv}
			\newenvironment{solution*}{\solutionEnv}{\endsolutionEnv}
		\else  
			
			\excludecomment{solution}
			\newenvironment{solution*}{}{}
			\excludecomment{solution*}
		\fi

	\ifdefined\showcomment 
		\newtheorem*{comment}{\normalfont\emph{Comment}}
	\else  
		\newenvironment{comment}{}{}
		\excludecomment{comment}
	\fi

\newcommand{\R}{\mathbb{R}}  
\newcommand{\N}{\mathbb{N}}  

\def\<{\left\langle}  
\def\>{\right\rangle}

\usepackage{multirow}
\usepackage{graphicx}  

\numberwithin{theorem}{subsection}
\numberwithin{equation}{subsection}

\renewcommand{\Pr}[1]{\mathbb{P}\left[#1\right]}
\newcommand{\Ex}[1]{\mathbb{E}\left[#1\right]}

\newcommand{\Var}[1]{\mathbb{V}\left[#1\right]}

\begin{document}

\title{Finite-sample properties of the trimmed mean}
\author{Roberto I. Oliveira\thanks{IMPA, Rio de Janeiro, RJ, Brazil. \texttt{rimfo@impa.br}. Supported by CNPq grants 432310/2018-5 (Universal) and 304475/2019-0 (Pro- dutividade em Pesquisa), and FAPERJ grants 202.668/2019 (Cientista do Nosso Estado) and 290.024/2021 (Edital Inteligência Artificial).},~ Paulo Orenstein\thanks{IMPA, Rio de Janeiro, Brazil. \texttt{pauloo@impa.br}. Supported by FAPERJ grant SEI- 260003/001545/2022. }~ and Zoraida F. Rico\thanks{University of Washington, Seattle, WA, USA. \texttt{zoraida@impa.br}}}
\date{} 
\maketitle

\begin{abstract}The trimmed mean of $n$ scalar random variables from a distribution $P$ is the variant of the standard sample mean where the $k$ smallest and $k$ largest values in the sample are discarded for some parameter $k$. In this paper, we look at the finite-sample properties of the trimmed mean as an estimator for the mean of $P$. Assuming finite variance, we prove that the trimmed mean is ``sub-Gaussian'' in the sense of achieving Gaussian-type concentration around the mean. Under slightly stronger assumptions, we show the left and right tails of the trimmed mean satisfy a strong ratio-type approximation by the corresponding Gaussian tail, even for very small probabilities of the order $e^{-n^c}$ for some $c>0$. In the more challenging setting of weaker moment assumptions and adversarial sample contamination, we prove that the trimmed mean is minimax-optimal up to constants.  
\end{abstract}


\section{Introduction}

We consider the fundamental problem of estimating the expectation of a one-dimensional random variable from an i.i.d.\ random sample. The sample mean is the standard estimator for this task. However, it can be very far from the best possible estimator when the data is (relatively) heavy-tailed or has outliers \cite{catoni2012challenging,dllo2016,lee2022optimal}. 

This paper studies the trimmed mean, a classical alternative to the sample mean. To define it, let $X_1,\dots,X_n$ be a random sample and denote by $X_{(1)}\leq X_{(2)}\leq \dots\leq X_{(n)}$ its order statistics. Given an integer $0\leq k<n/2$, the $k$-trimmed-mean of the sample is given by:
\[\overline{X}_{n,k}:=\frac{1}{n-2k}\sum_{i=k+1}^{n-k}X_{(i)}.\]
That is, $\overline{X}_{n,k}$ is the arithmetic mean of sample points after the $k$ largest and $k$ smallest values of the sample are removed. $\overline{X}_{n,k}$ equals the standard sample mean for $k=0$, whereas for $k=\lceil n/2\rceil-1$ it is a sample median. Intermediate choices of $k$ will lead to different trade-offs between bias and variance. 

Starting in the late Sixties, the asymptotic theory of the trimmed mean was analyzed in a number of papers \cite{Stigler1973,Jaeckel1971,Hall1981,LegerRomano1990,Jureckova1981} that are discussed in \S \ref{sub:trimmed} below. One focus of this literature is on the regime where $n\to +\infty$ and $k/n\to \eta \in (0,1)$; see, for instance, \cite{Stigler1973} for the asymptotic distribution of $\overline{X}_{n,k}$ in this setting.  

In this paper, we take a fresh look at the trimmed mean. Our main results are as follows:
\begin{itemize}
\item[\S \ref{sub:intro:subgaussian}] When the variance is finite, the trimmed mean is what is often called a ``sub-Gaussian estimator''~\cite{dllo2016} in the literature. Under mild additional conditions, this estimator has sharp constants and~``works''~for multiple confidence levels. Interestingly, these results are achieved by trimming a vanishing fraction of sample points. 
\item[\S \ref{sub:intro:CLT}] Under higher moment conditions, one can show that the trimmed mean satisfies a very strong form of the Central Limit Theorem, even relatively far in the tail of the distribution. This allows us to build $(1-\alpha)$-confidence intervals up to $\alpha=e^{-o(n^{c})}$ for some $c>0$. 
\item[\S \ref{sub:intro:minimax}] Additionally, the trimmed mean is minimax-optimal (up to constant factors) in settings allowing for heavier tails (e.g. possibly infinite variance) and adversarial data contamination.
\end{itemize}

We now discuss these findings in more detail.

\subsection{Sub-Gaussian properties.}\label{sub:intro:subgaussian}
Loosely speaking, a sub-Gaussian mean estimator $\hat{\mu}=\hat{\mu}(X_1,\dots,X_n)$ can estimate the mean $\mu$ of a random sample $X_1,\dots,X_n$ with Gaussian-type error bounds
\begin{equation}\label{eq:tailformulation}\Pr{|\hat{\mu}-\mu|\geq \frac{\sigma\,(c_1+c_2\,x)}{\sqrt{n}}}\leq C\,e^{-\frac{x^2}{2}},\end{equation}
for all $x>0$ in a suitable range, under the sole assumption that the variance $\sigma^2$ is finite; here, $c_1,c_2,C>0$ are universal constants independent of $\mu$, $\sigma^2$ or any other properties of the data generating mechanism. An equivalent formulation, which is perhaps more common in the literature, is that, for a given target confidence level $1-\alpha\in (0,1)$, the estimator should achieve \[\Pr{|\hat{\mu}-\mu|\leq \frac{\sigma\,(c_1 +c_2\sqrt{2\log(C/\alpha)})}{\sqrt{n}}}\geq 1-\alpha;\]
however, we will mostly work with the ``tail formulation''~given by  (\ref{eq:tailformulation}).

Catoni's seminal paper \cite{catoni2012challenging} seems to have been the first to pose the question of existence of sub-Gaussian estimators for finite samples. The paper shows that, while the sample mean is not sub-Gaussian for any nontrivial range of $x$, a suitable estimator achieves (\ref{eq:tailformulation}) with optimal constants for all $x=o(\sqrt{n})$ -- or equivalently, $\log(1/\alpha)=o(n)$ --, at least when $\sigma$ is known. The same paper shows that no estimator can achieve a value of $c_2$ smaller than $1$ (cf. Proposition 6.1). Later work proved positive and negative results about such estimators \cite{dllo2016} and obtained the optimal $c_2=1+o(1)$ for unknown variance \cite{lee2022optimal}. A series of papers by Minsker has looked at sub-Gaussian properties at variants of the so-called median-of-means construction \cite{MinskerUstat,minsker2021robust}. 

In what follows, we argue that the trimmed mean also achieves sub-Gaussian bounds. We first show that, in the most general setting, the trimmed mean is sub-Gaussian with suboptimal constants.

\begin{theorem}[Proof in \S \ref{subsub:proof:allsubgaussian}]\label{thm:allsubgaussian}Consider i.i.d.\ random variables $X_1,\dots,X_n$ with a well-defined mean $\mu$ and variance $\sigma^2<+\infty$. Take $0<x\leq \sqrt{{n}/{(\sqrt{2}+1)^2}-2}$ and consider the trimmed mean $\overline{X}_{n,k}$ with trimming parameter $k(x):=\lceil x^2/2\rceil$. Then: 
\[\Pr{|\overline{X}_{n,k(x)}-\mu|>\frac{(3\sqrt{2}+8+(4+4\sqrt{2})\,x)\,\sigma}{\sqrt{n}}}\leq 4\,\exp\left(-\frac{x^2}{2}\right).\]
\end{theorem}

In other words, the trimmed mean with the appropriate trimming parameter $k=k(x)$ achieves (\ref{eq:tailformulation}) with $C=4$ and $c_2 = 4\sqrt{2}+4$ and $c_1=3\sqrt{2}+8$. The next result shows we can reduce $c_1$ and improve $c_2$ to (nearly) optimal values under certain assumptions. This is the content of the next result. 

\begin{theorem}[Proof in \S \ref{subsub:proof:sharpersubgaussian}]\label{thm:sharpersubgaussian}Under the same assumptions as Theorem \ref{thm:allsubgaussian}, and given $a\in (0,1)$, we obtain the bound 
\[\Pr{|\overline{X}_{n,k(x)}-\mu|>\frac{\sigma\,(a\sqrt{2} + (1+a)\,x)}{\sqrt{n}}}\leq 4\,\exp\left(-\frac{x^2}{2}\right),\]
corresponding to $c_1=\sqrt{2}a$, $c_2=1+a$ in (\ref{eq:subgaussianguarantee}), 
under either one of the following additional assumptions:
\begin{enumerate}
\item $(1+x)/\sqrt{n}\leq \eta_F(a)$, where $\eta_F(a)$ depends only on $a$ and the common cumulative distribution function $F$ of the random variables $X_1,\dots,X_n$;
\item $\Ex{|X_1-\mu|^p}\leq (\kappa_{2,p}\sigma)^p$ for some $p>2$ and $\kappa_{2,p}<+\infty$, and additionally 
\begin{equation}\label{eq:conditiononasharper}a\geq 216\sqrt{2}\,\frac{(1+x)^2}{n} + 24\,{\kappa_{2,p}}\,\left(\frac{1+x}{\sqrt{n}}\right)^{\frac{p-2}{2p-2}}.\end{equation}
\end{enumerate}
\end{theorem}

Theorem \ref{thm:sharpersubgaussian} is closely related to recent work by Minsker \cite{minsker23a, MinskerUstat}. In our notation, these papers give conditions under which variants of the so-called ``median of means''~estimator achieves nearly optimal constants $c_2 = 1+o(1)$. For instance, Theorem 1 in \cite{minsker23a} achieves this under a variant of assumption 2 above. The main differences with our result are twofold. Firstly, the trimmed mean is easier to compute \cite[Remark 2, item (c)]{minsker23a}. Secondly, we obtain explicit finite-sample bounds on the relationship between $a$ and the quantities $\kappa_{2,p}$, whereas Minsker obtains asymptotic conditions for $a=o(1)$ \cite[Remark 2, item (b)]{minsker23a}. The constants in Theorem \ref{thm:sharpersubgaussian} are quite large to be practically meaningful; nevertheless, the trimmed mean with $k=k(x)$ has the fallback guarantee from Theorem \ref{thm:allsubgaussian} irrespective of any additional assumptions. 

We now investigate whether a choice of $k$ independent of $x$ is possible. This would be desirable in practice since the same trimming parameter would work for a range of $x$, or equivalently, for a range of confidence levels. 

In general, any estimator achieving sub-Gaussian bounds as in (\ref{eq:tailformulation}) must depend somehow on the desired confidence level (cf. \cite[Theorem 3.2, part 2]{dllo2016}). Therefore, we will need to make stronger assumptions to obtain a ``$x$-independent''~trimming parameter. 

The next result shows that weak higher-moment assumptions suffice for this purpose. In particular, we obtain ``multiple-$\delta$''~estimator in the language of \cite{dllo2016}, with optimal value $c_2=1$.    

\begin{theorem}[Proof in \S \ref{subsub:proof:multiplesubgaussian}]\label{thm:multiplesubgaussian} Let $X_1,\dots,X_n$ be i.i.d.\ random variables with well-defined mean $\mu$, finite variance $\sigma^2$, and such that
\[\exists p>2,\kappa_{2,p}\geq 1\,:\, \Ex{|X_1-\mu|^p}\leq (\kappa_{2,p}\sigma)^p.\]
Assume that $k_*\in\N$ satisfies
\[432\,\frac{k^{3/2}_*}{n} + 24\sqrt{2}\,\kappa_{2,p}\,\frac{k^{\frac{3p-4}{4p-4}}_*}{n^{\frac{p-2}{4p-4}}}\leq 1.\]
Then for any $0<x\leq \sqrt{2k_*}$,
\[\Pr{|\overline{X}_{n,k_*}-\mu|>\frac{(1+x)\,\sigma}{\sqrt{n}}}\leq 4\,\exp\left(-\frac{x^2}{2}\right).\]\end{theorem}

As a consequence, if $n\to +\infty$ with $\kappa_{2,p}:=\kappa^{(n)}$ possibly varying with $n$, and choosing 
\[k^{(n)}_* = o\left(\frac{n^\frac{p-2}{3p-4}}{(\kappa^{(n)})^\frac{4p-4}{p-2}}\right),\]
one obtains
\[\Pr{|\overline{X}_{n,k^{(n)}_*}-\mu|>\frac{(1+\sqrt{2\log(4/\alpha)})\,\sigma}{\sqrt{n}}}\geq 1-\alpha\mbox{ for }0<\log(4/\alpha)\leq k_*.\]
This result can be compared with \cite[Theorem 3.2]{dllo2016}, which only covers the case of finite kurtosis (i.e., $p=4$), but allows for a wider range of $\alpha$ which is roughly
$\log(4/\alpha)\ll o((n/\kappa^{(n)})^{2/3})$ (whereas Theorem \ref{thm:multiplesubgaussian} requires $\log(4/\alpha)\ll o(n^{1/4}/(\kappa^{(n)})^{4})$).

\subsection{Precise Gaussian approximation and confidence intervals}\label{sub:intro:CLT} So far, our results have presented various sub-Gaussian concentration bounds for the trimmed mean. While theoretically interesting, it is known that such concentration bounds are often pessimistic. 

To mitigate this limitation, some papers have tried to show that certain sub-Gaussian estimators are asymptotically efficient, which indicates that their practical performance may be closer to ideal. For instance, \cite{minsker2021robust} establishes the asymptotic statistical efficiency of certain robust mean estimator. It is not hard to show a similar result for the trimmed mean when $k$ is fixed; see \S \ref{sub:asymptoticefficiency} in the Appendix for details. 

In what follows, we show that the trimmed mean satisfies a type of Gaussian approximation even very far into the tail of its distribution. This can be seen as a strengthening of results known for other estimators. 
\begin{theorem}[Proof in \S \ref{sub:proof:preciseconfidence}]\label{thm:preciseconfidence}There exists a constant $C>0$ such that the following holds. Let $X_1,\dots,X_n$ be i.i.d.\ random variables with mean $\mu\in\R$, variance $\sigma^2\in (0,+\infty)$ and such that $\Ex{|X_1-\mu|^p}\leq (\kappa_{2,p}\sigma)^p$ for some $p>2$ and $\kappa_{2,p}<+\infty$. Given $x>0$, $\delta>0$, and a trimming parameter \[k_*\geq \max\left\{2,\left\lceil \log\left(\frac{4}{\delta\,(1-\Phi(x))}\right)\right\rceil\right\}\] 
satisfying
\[\gamma:=\kappa_{2,p}\frac{k_*^{2-\frac{p}{4p-4}}}{n^{\frac{p-2}{4p-4}}}<\frac{1}{C},\]
we have that 

\begin{eqnarray*}\left|\frac{\Pr{|\overline{X}_{n,k_*}-\mu|>\frac{x\sigma}{\sqrt{n}}}}{1-\Phi(x)} - 1\right| &\leq & C\gamma+\delta;\\  \left|\frac{\Pr{|\overline{X}_{n,k_*}-\mu|>\frac{x\hat{\sigma}_{n,k_*}}{\sqrt{n}}}}{1-\Phi(x)} - 1\right| &\leq & C\gamma+\delta; \\ \frac{\Pr{\hat{\sigma}_{n,k_*}>(1+\gamma)\sigma}}{1-\Phi(x)} &\leq & \delta,\end{eqnarray*}
where 
\[\hat{\sigma}_{n,k_*}^2:=\frac{1}{n-2k_*}\sum_{i=k_1+1}^{n-k_2}\,(X_{(i)} - \overline{X}_{n,k_*})^2\]
is the empirical variance of the trimmed sample.\end{theorem}

It is instructive to compare this result with Berry-Ess\'{e}en-type inequalities for the sample mean, which give bounds such as 
\[\left|{\Pr{|\overline{X}_{n}-\mu|>\frac{x\sigma}{\sqrt{n}}}} - ({1-\Phi(x)})\right|\leq \frac{C\kappa^3_{2,3}}{\sqrt{n}}.\]
Our theorem is a stronger result than such an additive probability approximation as soon as $(1-\Phi(x))\ll n^{-1/2}$. In particular, it gives strong bounds even when the corresponding Gaussian probabilities are quite small. 

Our result also differs from self-normalized inequalities for the sample mean. For instance,  \cite{Jing2003} proves a result of the form: 
\begin{equation}\label{eq:jing}\left|\frac{\Pr{|\overline{X}_{n}-\mu|>\frac{x\hat{v}_{n}}{\sqrt{n}}}}{1-\Phi(x)} - 1\right|\ll 1\mbox{ with }\hat{v}^2_{n} := \frac{1}{n}\sum_{i=1}^n(X_i-\mu)^2,\mbox{ whenever }0\leq x \ll \frac{n^{\frac{p-2}{2p}}}{\kappa_{2,p}}.  \end{equation} 
This result works for the sample mean, and allows for a broader  range of $x$ than our Theorem \ref{thm:preciseconfidence}. However, the intuitive reason why (\ref{eq:jing}) works is the appearance of the self-normalized ratio $(\overline{X}_n-\mu)/\hat{v}_n$, whereby large values in the sample are compensated by a large value of $\hat{v}_n$. In particular, self-normalized Gaussian bounds do {\em not} imply good concentration for $(\overline{X}_n-\mu)/\sigma$ because the probability that $\hat{v}_n/\sigma$ is large may be nonnegligible when compared to $1-\Phi(x)$. This also implies that confidence intervals built via (\ref{eq:jing}) may be much wider than what one would expect from the Central Limit Theorem. 

By contrast, Theorem \ref{thm:preciseconfidence} guarantees that $(\overline{X}_{n,k} - \mu)/\sigma$ is well behaved, and that $\hat{\sigma}_{n,k}/\sigma$ is close to $1$. As a result, the confidence intervals one may obtain for $\overline{X}_{n,k}$ have essentially the length predicted by the CLT, even when the desired confidence level is very close to $1$. The following asymptotic result illustrates this point. 

\begin{corollary}[Proof omitted] Assume that $\{X_{i}\}_{i=1}^{+\infty}$ is an i.i.d.\ sequence of random variables with mean $\mu$, variance $\sigma^2\in (0,+\infty)$ and $\Ex{|X_1-\mu|^p}<+\infty$ for some $p>2$. Let $\{\alpha_n\}_{n\in\N}\subset (0,1/2)$ and $k_n\in \N$ satisfy 
\[k_n - \log(1/\alpha_n)\to +\infty\mbox{ and }\frac{k_n}{n^{\frac{p-2}{7p-8}}}\to 0\mbox{ as }n\to +\infty.\]
Define  $x_n\in\R_+$ via $(1-\Phi(x_n))=\alpha_n$. Then:
\[\Pr{|\overline{X}_{n,k_n}-\mu|\leq \frac{x_n\sigma}{\sqrt{n}}} = 1- (1+o(1))\,\alpha_n\]
and
\[\Pr{|\overline{X}_{n,k_n}-\mu|\leq \frac{x_n\hat{\sigma}_{n,k_n}}{\sqrt{n}} \mbox{ and }\hat{\sigma}_{n,k_n}\leq (1+o(1))\sigma} = 1- (1+o(1))\,\alpha_n,\]
where the $o(1)$ terms go to $0$ as $n\to +\infty$.\end{corollary}

\begin{remark} It would also be possible to show that $(\overline{X}_{n,k}-\mu)/\sigma$ converges weakly to a standard normal when $k_n\ll n^{(p-2)/(5p-6)}$ and $\nu_{p} = \Ex{|X_1-\mu|^p}^{1/p}<+\infty$ for some $p>2$. This contrasts with the case of $k_n/n\to \eta>0$, which may or may not lead to Gaussian limits \cite{Stigler1973}. \end{remark}

\subsection{Heavier tails and contamination}\label{sub:intro:minimax} We now move away from the sub-Gaussian setting in two ways. First, we do not necessarily assume that the distribution of the $X_i$ has finite variance. Secondly, we allow for adversarial sample contamination \cite{diakonikolas2019}, whereby an $\epsilon$-fraction of sample points can be arbitrarily corrupted. In this setting, we obtain the following result.

\begin{theorem}[Proof in Section \ref{sec:proof:minimaxcontaminated}]\label{thm:minimaxcontaminated} Let $X_1,\dots,X_n$ be i.i.d.\ with well defined mean $\mu=\Ex{X_1}$ and define $\nu_p^p:=\Ex{|X_1-\mu|^p}$ (for $p\geq 1$). Take $\epsilon\in (0,1/2)$. Let  $X_1^\epsilon,\dots,X_{n}^{\epsilon}$ be an $\epsilon$-contamination of $X_1,\dots,X_n$, in the sense that:
\[\#\{i\in[n]\,:\,X^\epsilon_i\neq X_i\}\leq \epsilon\,n,\]
and let $\overline{X^\epsilon}_{n,k}$ denote the trimmed mean computed on the contaminated sample. 

Fix $0<\alpha<1$ and assume that \[(\sqrt{2\lfloor \epsilon n\rfloor + 2\lceil \log(4/\alpha)\rceil-1} + \sqrt{\lceil \log(4/\alpha)\rceil})^2 \leq dn\mbox{ for some }0<d<1.\] Then the trimmed mean estimator with parameter \[k:=\lfloor \epsilon n\rfloor + \lceil  \log(4/\alpha)\rceil\]
satisfies the following bound:
\begin{equation}\label{eq:proofminimaxmomentscontaminated}\Pr{|\overline{X^\epsilon}_{n,k}-\mu|\leq C(d)\,\left(\inf_{p\geq 1}\nu_p\,\epsilon^{\frac{p-1}{p}} + \inf_{1\leq q\leq 2}\nu_q\,\left(\frac{\log(4/\alpha)}{n}\right)^{\frac{q-1}{q}}\right)}\geq 1-\alpha,\end{equation}
where $C(d)$ depends on $d$ only.
\end{theorem}

Theorem \ref{thm:minimaxcontaminated} shows that, with probability $1-\alpha$, the error of trimmed mean of the contaminated sample consists of two terms: a {\em random fluctuations term} and a contamination term,
\[ C(d)\inf_{1\leq q\leq 2}\nu_q\,\left(\frac{\log(4/\alpha)}{n}\right)^{\frac{q-1}{q}}\mbox{ and }C(d)\inf_{p\geq 1}\nu_p\,\epsilon^{\frac{p-1}{p}}, \mbox{ respectively.}\]
It turns out that both terms are minimax-optimal up to the constant factor $C(d)$. This follows from lower bounds in \cite{dllo2016} and \cite{minsker2018} that we recall in \S \ref{sub:minimaxmoments} in the Appendix. As far as we know, the trimmed mean is the only estimator of one-dimensional expectations satisfying this property.

\begin{remark}One drawback of Theorem \ref{thm:minimaxcontaminated} is that the choice of trimming parameter $k$ requires a choice of confidence level $1-\alpha$ and knowledge of the contamination parameter $\epsilon$. We have already noted that the choice of $1-\alpha$ is unavoidable, but the need to know $\epsilon$ is less clear.\end{remark}

\begin{remark}
    The recent results of Oliveira and Resende \cite{oliveira2023} on ``trimmed empirical processes''~also imply a form of Theorem \ref{thm:minimaxcontaminated}. However, Theorem \ref{thm:minimaxcontaminated} was obtained first, and relies on a different proof technique, discussed below, that allows for more precise results such as Theorem \ref{thm:preciseconfidence}. 
\end{remark}

\subsection{Technical and conceptual contributions} The main idea behind all results in this paper is to look at the trimmed mean as an average of {\em conditionally} i.i.d.\ random variables. This is most easily seen when the $X_i$ have an atom-free distribution $P$. In this case, conditioning on $X_{(k)}=x$, $X_{(n-k+1)}=y$ makes the random sample $\{X_{i}\,:\, i\in [n], x<X_i<y\}$ i.i.d.\ from a compactly supported distribution $P^{(x,y)}$ with certain mean and centered moment parameters. The proof of Theorems \ref{thm:allsubgaussian} and \ref{thm:minimaxcontaminated} then consists of applying Bernstein's concentration inequality conditionally, and checking that the mean and other parameters of $P^{(x,y)}$ are not too far from those of $P$. 

The Gaussian approximation in Theorem \ref{thm:preciseconfidence} requires a slightly different approach where we apply the self-normalized CLT of Jing et al. \cite{Jing2003}, quoted in (\ref{eq:jing}), to the trimmed sample. As noted above, this introduces a self-normalized ratio-type quantity. However, because our random sample is {\em bounded}, we can show that the denominator in this sample concentrates around its expectation under $P^{(x,y)}$. In this way, we obtain a conditional variant of (\ref{eq:jing}) where the denominator of the ratio is nonrandom. 

Besides the many calculations needed to make everything work, there are two ways in which the above outline differs from our actual proofs. The first one is that the sample distribution need not be atom-free. We circumvent this by using {\em quantile transforms}, whereby $X_i=F^{-1}(U_i)$ for uniform random variables $\{U_i\}_{i\in[n]}$. The second one is that, when considering contamination, we will need to bound the trimmed mean on the contaminated sample by an asymmetrically trimmed mean on the ``clean''~sample; see Proposition \ref{prop:effectcontamination} for details.

\subsection{Additional background}

\subsubsection{Background on the trimmed mean.} \label{sub:trimmed}

The literature on the trimmed mean is quite large, and we present a brief and partial review. 

Huber \cite{Huber1972} gives early historical references for the trimmed mean. Tukey's seminal paper \cite{Tuckey1962} explicitly proposes the trimmed mean and the related Winsorized mean as ways to estimate location parameters from outlier-contaminated data. Tukey also suggested the possibility of data-dependent choices of the trimming parameter $k$.

The trimmed mean was a popular topic of study in classical Robust Statistics. Stigler \cite{Stigler1973} gives the asymptotic distribution of $\overline{X}_{n,k}$ when $n\to +\infty$ and $k = \lfloor \eta n\rfloor$, which may or may not be Gaussian. This is in contrast with most of our results, where the trimming parameter is sublinear in $n$.

 One problem we do not consider is how to choose $k$ adaptively. Starting with Jaeckel \cite{Jaeckel1971}, a number of papers have appeared on this topic \cite{Hall1981,Jureckova1981,LegerRomano1990,Jian1996,Lee2004}. The theory in these papers requires much stronger assumptions than we do, including symmetry of the distribution around the median, and some kind of ``good behavior''~of the data generating distribution. 

Experiments on trimmed means are presented in many papers. Hogg \cite{Hogg1974} presents a number of results on adaptive robust estimators and makes concrete suggestions on trimmed means. Experiments comparing trimming and winsorization in \cite{Dixon1974} suggest that trimming is usually better. Stigler \cite{Stigler1977} compares different robust estimators on real datasets and shows that the trimmed mean with $k=\lfloor 0.1n\rfloor$ is often one of the very best estimators. Further analysis by Rocke et al. \cite{Rocke1982} does not quite corroborate Stigler, but still indicates that the trimmed mean has good performance. 

Finally, we note in passing that there are papers on high-dimensional versions of the trimmed mean and related estimators \cite{Maller1988,lugosi2021,oliveira2023}.

\subsubsection{Finite-sample bounds, sub-Gaussian estimators and related topics}\label{sub:optimal}

Concentration inequalities for sums of bounded independent random variables are a classical topic covered in \cite{boucheron2013concentration} and many other references. Finite-sample self-normalized concentration and Gaussian approximations are discussed in the survey by Shao and Wang \cite{Shao2013} and in the book by de la Pe\~{n}a, Lai and Shao \cite{delaPea2009}, among other places.

More recently, there has been interest in {\em designing estimators with optimal concentration properties}. In the so-called ``sub-Gaussian'' case, one is interested in finding, for each sample size $n$ and confidence level $1-\alpha$, an estimator $\widehat{E}_{n,\alpha}:\R^n\to\R$ with the following property. Let $X_1,\dots,X_n$ be an i.i.d.\ sample from an unknown distribution with mean $\mu$ and finite variance $\sigma^2$. Then:
\begin{equation}\label{eq:subgaussianguarantee}\Pr{|\widehat{E}_{n,\alpha}(X_1,\dots,X_n) - \mu|\leq C\sigma\,\sqrt{\frac{\log(1/\alpha)}{n}}}\geq 1-\alpha,\end{equation}
where $C>0$ is a universal constant. While the estimator may depend on $\alpha$ as well on $n$, the above bound should hold {\em uniformly} over all distributions with finite second moments, irrespective of how heavy their tails are. The sample mean will not achieve such a bound for any small enough $\alpha$.  

Catoni's seminal work \cite{catoni2012challenging} provides one such estimator, with nearly optimal $C=\sqrt{2}+o(1)$ in the case where $\sigma^2$ is known and $\log(1/\alpha)\ll n$. Recent work \cite{lee2020} gives sub-Gaussian estimators with near optimal $C=\sqrt{2}+o(1)$ for the case of unknown variance. \cite{dllo2016} explores the notion of sub-Gaussian estimators in greater depth: it shows for instance that sub-Gaussian estimators must indeed depend on the desired confidence $1-\alpha$, and that some bound of the sort $\log(1/\alpha)\leq c\,n$ is needed. There has been great interest in extending these results to higher dimensions: see \cite{lugosi2019,lugosi2021} and the survey \cite{lugosi2019B} for more details.
 
Some papers consider what happens when the variance may be infinite, and we only assume $\Ex{|X_1-\mu|^p}\leq \nu_p^p$ for some $1<p<2$. It follows from \cite{bubeck2013} that the so-called median of means estimator satisfies
\begin{equation}\label{eq:finitemomentguarantee}\Pr{|\widehat{E}_{n,\alpha}(X_1,\dots,X_n) - \mu|\leq C\nu_p\,\left(\frac{\log(1/\alpha)}{n}\right)^{1-1/p}}\geq 1-\alpha,\end{equation}
for some universal $C>0$. It is possible to show that this cannot be improved, up to the value of $C$ \cite[Theorem 3.1]{dllo2016}. The upshot is that the median-of-means estimator is optimal for any choice of $1\leq p\leq 2$. Our Theorem \ref{thm:minimaxcontaminated} gives a similar bound for the trimmed mean. 

\subsubsection{Adversarial contamination} 

Finally, we discuss the model of adversarial data contamination. Recall that the traditional contamination model in Robust Statistics is that of Huber \cite{Huber1964}, where there is an uncontaminated distribution $P$, but data comes from a contaminated law $(1-\epsilon)P + \epsilon Q$, with $Q$ unknown. In the adversarial model we consider, an $\epsilon$ fraction of data points may be replaced {\em arbitrarily}. In particular, one may imagine that an adversary gets to see the uncontaminated random sample and then chooses which points to replace so as to foil the statistician. This model has become standard in recent work on algorithmic high-dimensional Statistics \cite{diakonikolas2019}. This model places strong requirements on an estimator, which can sometimes simplify proving theoretical results.

\subsection{Organization}

The remainder of the paper is organized as follows. Section \ref{sec:assump} fixes notation and records some facts we will need later in the text. The trimmed mean is introduced in a somewhat more general form in Section \ref{sec:TMgeneral}. There, we also look at its conditional distribution, and prove that it behaves nicely under contamination. Section \ref{sec:boundstrimmed} gives a number of bounds on parameters pertaining to the conditional distribution of the trimmed mean. The sub-Gaussian concentration results discussed in \S \ref{sub:intro:subgaussian} are proven in Section \ref{sec:subGaussian}. Section \ref{sec:CLT} proves results on Gaussian approximation and confidence intervals that were stated in \S \ref{sub:intro:CLT}. The minimax result in \S \ref{sub:intro:minimax} is proven in Section \ref{sec:proof:minimaxcontaminated}, and Section \ref{sec:experiments} presents a small set of illustrative experiments. The Appendix contains some technical estimates and additional observations. 

\section{Preliminaries}\label{sec:assump}

\subsection{General notation} In this paper, $\N:=\{1,2,3,\dots\}$ is the set of positive integers. Given $n\in\N$, $[n]:=\{i\in\N\,:\,i\leq n\}$ is the set of numbers from $1$ to $n$. For a real number $x$, 
$\lfloor x\rfloor$ and $\lceil x\rceil$ denote the floor and ceiling of $x$, respectively. The cardinality of a finite set $A$ is denoted by $\# A$. Given sequences $\{a_n\}_{n\in\N}$, $\{b_n\}_{n\in\N}$ of positive real numbers, we write $a_n\ll b_n$ or $a_n=o(b_n)$ when $a_n/b_n\to 0$. 

\subsection{Probability notation and facts.}\label{sub:probnotation}

The mean (expectation) and variance of a real-valued random variable $Z$ are denoted by $\Ex{Z}$ and $\Var{Z}$, respectively. We use ``i.i.d.''~for ``independent and identically distributed.'' 

In the entire paper, $X_1,\dots,X_n$ will be i.i.d.\ random variables with a well-defined mean $\mu=\Ex{X_1}$ and a cumulative distribution function $F(t):=\Pr{X_1\leq t}$ ($t\in \R$). We will use the notation $\nu_p:=(\Ex{|X_1-\mu|^p})^{1/p}$ ($p\geq 1$) for the centered absolute $p$-th moment of $X_1$, and also $\sigma:=\nu_2$.
\[F^{-1}(q):=\inf\{t\in\R\,:\, F(t)\geq q\}\,\,(q\in (0,1)).\]
is the generalized inverse (or quantile transform) of $F$. We note the following straightforward fact. 

\begin{proposition}\label{prop:rhop}Assume that $p\geq 1$ and $\nu_p<+\infty$. For $\xi\in (0,1)$, define 
\begin{equation}\rho_{F,p}(\xi):=\sup\left\{\frac{\Ex{|X_1-\mu|^p\,Z}^{\frac{1}{p}}}{\nu_p}\,:\, 0\leq Z\leq 1\mbox{ random variable with }\Ex{Z}\leq \xi\right\}\end{equation}\label{eq:rhoFp}
when $\nu_p>0$, or $\rho_{F,p}(\xi)=0$ otherwise. Then $0\leq \rho_{F,p}\leq 1$, 
$\lim_{\xi\to 0}\rho_{F,p}(\xi)=0$ and 
\[\forall r>0\,:\,\Pr{|X_1-\mu|>r\nu_p}\leq \frac{\rho_{F,p}\left(\frac{1}{r^p}\right)^p}{r^p}.\]
Finally, if $\nu_q\leq \kappa_{p,q}\,\nu_p<+\infty$ for some $q>p$, then $\rho_{F,p}(\xi)\leq \kappa_{p,q}\,\xi^{\frac{1}{p}-\frac{1}{q}}$.\end{proposition}
\begin{proof}The facts that $0\leq \rho_{F,p}\leq 1$ and $\lim_{\xi\to 0}\rho_{F,p}(\xi)=0$ are straightforward. 

For the probability bound, we note that 
\begin{equation}\label{eq:probtobeimproved}\Pr{|X_1-\mu|>r\nu_p}\leq \frac{\Ex{|X_1-\mu|^p{\bf 1}_{\{|X_1-\mu|>r\nu_p\}}}}{(\nu_p r)^p}.\end{equation}
Omitting the indicator in the RHS, we see that
\[\Pr{|X_1-\mu|>r\nu_p}\leq \frac{1}{r^p}.\]
This implies that $Z:={\bf 1}_{\{|X_1-\mu|>r\nu_p\}}$ satisfies $0\leq Z\leq 1$ and $\Ex{Z}\leq r^{-p}$. Therefore, \[\Ex{|X_1-\mu|{\bf 1}_{\{|X_1-\mu|>r\nu_p\}}}\leq \nu^p_p\,\rho_{F,p}\left({\frac{1}{r^p}}\right)^p\] and we can plug this back into (\ref{eq:probtobeimproved}) to finish the proof.  

Finally, if $\nu_q\leq \kappa_{p,q}\,\nu_p<+\infty$, then H\"{o}lder's inequality implies that, for any $0\leq Z\leq 1$ with $\Ex{Z}\leq \xi$
\[\Ex{|X_1-\mu|^pZ}\leq (\Ex{|X_1-\mu|^q})^{\frac{p}{q}}\,\Ex{Z}^{1-\frac{p}{q}}\leq \kappa_{p,q}^{p}\,\nu^p_p\,\xi^{1-\frac{p}{q}},\]
from which $\rho_{F,p}(\xi)^p\leq \kappa_{p,q}^{p}\,\xi^{1-\frac{p}{q}}$ follows. \end{proof}

\subsection{Concentration and Gaussian approximation for i.i.d.\ sums}
We record here two facts about sums of bounded i.i.d.\ random variables. The first one is the classical Bernstein's inequality, proven in e.g. \cite[eq. (2.10)]{boucheron2013concentration}.

\begin{theorem}[Bernstein's inequality]\label{thm:bernstein}Consider i.i.d.\ random variables 
\[\{Z_i\}_{i=1}^n\mbox{ with $\Ex{Z_1}=\mu_Z,\,\Var{Z_1}\leq \sigma^2_Z$ and $|Z_1-\mu|\leq \Delta_Z$ almost surely.}\]
Let
\[\overline{Z}_n:=\frac{1}{n}\sum_{i=1}^nZ_i.\]
Then, for all $z>0$,
\[\Pr{\overline{Z}_n - \mu\geq \frac{z\sigma_Z}{\sqrt{n}} + \frac{z^2\,\Delta_Z}{12n}}\leq \exp\left(\frac{-z^2}{2}\right).\] 
\end{theorem}

We will also need the following special case of the self-normalized Central Limit Theorem of Jing, Shao and Wang \cite[Theorem 2.1]{Jing2003}. In what follows, $\Phi$ is the standard Gaussian cumulative distribution function. 
\begin{theorem}[\cite{Jing2003}]\label{thm:CLTJing}  There exists a constant $A>0$ such that the following holds. Consider i.i.d.\ random variables 
\[\{Z_i\}_{i=1}^n\mbox{ with $\Ex{Z_1}=\mu_Z,\,\Var{Z_1}=\sigma^2_Z$ and $|Z_1-\mu|\leq \Delta_Z$ a.s.}.\]
Define $\overline{Z}_n$ as in Theorem \ref{thm:bernstein} and
\[V_n:=\sqrt{\frac{1}{n}\sum_{i=1}^n(Z_i-\mu_Z)^2}.\]
 Given $x\in\R$, if 
\[r(x,\Delta,\sigma,n):=(1+\max\{x,0\})^3\frac{\Delta_Z}{\sigma_Z\sqrt{n}}\leq \frac{1}{A},\]
we have
\[\Pr{\overline{Z}_n- \mu_Z\geq \frac{x\,V_n}{\sqrt{n}}}= \left(1 + \eta\right)\,(1-\Phi(x))\]
with $|\eta|\leq Ar(x,\Delta_Z,\sigma_Z,n)$.
\end{theorem}
\begin{proof}[Proof sketch] For $x\geq 0$, this follows from Theorem 2.1 in \cite{Jing2003} if one notes that
\[\Ex{(Z_1-\mu_Z)^2{\bf 1}_{\{|Z_1-\mu_Z|>\sigma_Z\sqrt{n}/(1+x)\}}}\leq \frac{1+x}{\sigma_Z\sqrt{n}}\,\Ex{|Z_1-\mu_Z|^3\,{\bf 1}_{\{|Z_1-\mu_Z|>\sigma_Z\sqrt{n}/(1+x)\}}}\]
and also $\Ex{|Z_1-\mu_Z|^3}\leq \Delta_Z \Ex{(Z_1-\mu_Z)^2} = \Delta_Z\sigma_Z^2$. For $x<0$, the above is a consequence of standard Berry-Ess\'{e}en bounds along with the fact that $1-\Phi(x)\geq 1/2$. \end{proof}

\section{Trimmed means: first steps}\label{sec:TMgeneral}

In this section, we define the trimmed mean and study its basic distributional properties. Throughout this section, $X_1,\dots,X_n$ are i.i.d.\ real-valued random variables with common cumulative distribution function $F$, and $X_{(1)}\leq X_{(2)}\leq \dots \leq X_{(n)}$ are the order statistics of the $X_i$. 

\subsection{Definitions} As our first step, we define the trimmed mean and related quantities. It will be important for later applications to define an asymmetrical trimmed mean where we may remove different numbers of points from the two tails of the distribution. 

\begin{definition}[$(k_1,k_2)$-trimmed mean,  variance and width]\label{def:TM} Let $X_1,\dots,X_n$ be i.i.d.\ random variables with common distribution $P$ and distribution function \(F(t):=P(-\infty,t]\,(t\in\R).\) Let \[X_{(1)}\leq X_{(2)}\leq \dots \leq X_{(n)}\] denote the increasing rearrangement of the sample (i.e., its order statistics). Assume $k_1,k_2\in\N\cup\{0\}$ satisfy $k_1+k_2<n$. The $(k_1,k_2)$-trimmed mean estimator is defined as
\[\overline{X}_{n,k_1,k_2}:=\frac{1}{n-k_1-k_2}\sum_{i=k_1+1}^{n-k_2}X_{(i)}\]
and the $(k_1,k_2)$-trimmed variance estimator is 
\[\hat{\sigma}^2_{n,k_1,k_2}:=\frac{1}{n-k_1-k_2}\sum_{i=k_1+1}^{n-k_2}\,(X_{(i)} - \overline{X}_{n,k_1,k_2})^2.\] The $(k_1,k_2)$-width is defined as $\Delta_{n,k_1,k_2}:=X_{(n-k_2-1)} - X_{(k_1)}$. When $k_1=k_2=k$, we write $\overline{X}_{n,k}$ for $\overline{X}_{n,k,k}$, and similarly for the other quantities. \end{definition}

\subsection{Distributional properties}

A simple, but crucial observation about the trimmed mean is that, under a certain conditioning, it is an i.i.d.\ sum. This is easier to see when $F$ is continuous: conditionally on $X_{(k_1)}$ and $X_{(n-k_2)},$ the random variables $X_{(i)}$ with $k_1+1\leq i\leq n-k_2+1$ are i.i.d.

For general $F$, we use the quantile transform $F^{-1}$ (defined in \S \ref{sub:probnotation}) to arrive at a similar result. We start with the following proposition.  

\begin{proposition}\label{prop:richer} If $\{X_i\}_{i=1}^n$ and $F$ are as above,  one can define (on a richer probability space, if needed) random variables $U_1,\dots,U_n$ that are i.i.d.\ uniform over $(0,1)$, such that $F^{-1}(U_i)=X_i$ and $F^{-1}(U_{(i)})=X_{(i)}$ almost surely for each $i\in[n]$.\end{proposition}
\begin{proof}[Proof sketch] This result can be proven by recalling that $X_1$ has the same law as $F^{-1}(U_1)$, using that $F^{-1}$ is monotone non-increasing and applying a coupling argument; we omit the details.\end{proof}

Given this construction, we define the conditional mean and variance parameters associated with the random variables $F^{-1}(U_i)$.

\begin{definition}[$(a,b)$-trimmed population parameters]\label{def:trimmedpop} Let $F^{-1}$ be the quantile transform of $F$. Given $0<a<b<1$, we define $P^{(a,b)}$ as the distribution of $F^{-1}(U^{(a,b)})$, where $U^{(a,b)}$ is uniform over $(a,b)$. The $(a,b)$-trimmed population mean and variance are (respectively) the mean and variance of this distribution: \[\mu^{(a,b)}:=\Ex{F^{-1}(U^{(a,b)})} = \frac{1}{b-a}\int_{a}^bF^{-1}(u)\,du \] and \[\left(\sigma^{(a,b)}\right)^2:= \Var{F^{-1}(U^{(a,b)})}= \frac{1}{b-a}\int_{a}^b(F^{-1}(u)-\mu(a,b))^2\,du.\]
We also define the $(a,b)$-trimmed width as
$\Delta^{(a,b)}:=F^{-1}(b)-F^{-1}(a)$, noting that $P^{(a,b)}$ is supported on an interval of size $\Delta^{(a,b)}$. Finally, we set
\[\nu_p^{(a,b)}:=\Ex{\left|F^{-1}(U^{(a,b)})-\mu^{(a,b)}\right|^p}^{\frac{1}{p}}\,\,(p\geq 1).\]

\end{definition}

The next result is an easy consequence of the above discussion. 

\begin{corollary}\label{cor:conditionalsample}In the setting of Proposition \ref{prop:richer}, it holds that, conditionally on $U_{(k_1)} = a<U_{(n-k_2+1)}=b$ (with $0<a<b<1$), the random variables $X_{(i)}$ with $k_1+1\leq i\leq n-k_2$ are (up to their ordering) i.i.d.\ with common law $P^{(a,b)}$. \end{corollary}
\begin{proof}Under this conditioning, the random variables $U_{(i)}$ with $k_1+1\leq i\leq n-k_2$ are (up to their ordering) i.i.d.\ uniform over $(a,b)$. Since $X_{(i)}=F^{-1}(U_{(i)})$ for each $i\in[n]$, the result follows.\end{proof}

\subsection{The case of contaminated data} In Section \ref{sec:proof:minimaxcontaminated}, we apply the trimmed mean to adversarially contaminated data. In this setting, we introduce the notation specific to this case and explain why asymmetrically trimmed means behave well under this sort of contamination. 

We start with a definition.

\begin{definition}[Contamination and trimmed mean] Random variables $X^\epsilon_1,\dots,X^\epsilon_n$ are an $\epsilon$-contamination of the i.i.d.\ sample $X_1,\dots,X_n$ if
\[\#\{i\in[n]\,:\, X_i^\epsilon\neq X_i\}\leq \epsilon n.\]
Letting $X^\epsilon_{(1)}\leq X^\epsilon_{(2)}\leq \dots \leq X^\epsilon_{(n)}$ denote the order statistics of the random sample, and given $(k_1,k_2)\in\N^2$ with $k_1+k_2<n$, we define the $\epsilon$-contaminated $(k_1,k_2)$-trimmed mean as follows: \[\overline{X^\epsilon}_{n,k_1,k_2}:=\frac{1}{n-k_1-k_2}\sum_{i=k_1+1}^{n-k_2}X^\epsilon_{(i)}.\]\end{definition}

Unlike with $\overline{X}_{n,k_1,k_2}$, there is no way to represent $\overline{X^\epsilon}_{n,k_1,k_2}$ as an average of conditionally i.i.d.\ random variables. In fact, the contaminated trimmed mean can be quite bad when $\min\{k_1,k_2\}<\lfloor\epsilon n\rfloor$, as a suitable contamination can drive the value of $|\overline{X^\epsilon}_{n,k_1,k_2}|$ to $+ \infty$. On the other hand, if $\min\{k_1,k_2\}\geq \lfloor\epsilon n\rfloor$, one can relate $\overline{X^\epsilon}_{n,k_1,k_2}$ to trimmed mean over the clean sample $X_1,\dots ,X_n$. 

\begin{proposition}\label{prop:effectcontamination}Assume $k_1,k_2\in\N$ satisfy $\min\{k_1,k_2\}\geq \lfloor \epsilon n\rfloor$ and $k_1+k_2<n$ Then:\[\overline{X}_{n,k_1-\lfloor \epsilon n\rfloor,k_2+\lfloor \epsilon n\rfloor}\leq \overline{X^\epsilon}_{n,k_1,k_2}\leq \overline{X}_{n,k_1+\lfloor \epsilon n\rfloor,k_2-\lfloor \epsilon n\rfloor}.\]
\end{proposition}

\begin{proof}It suffices to prove the following
\begin{equation}\label{eq:claimorderstats}\mbox{{\bf Claim:} }\forall i\in \{k_1+1,\dots,n-k_2\}\,:\, X_{(i-\lfloor \epsilon n\rfloor)}\leq X^\epsilon_{(i)}\leq X_{(i+\lfloor \epsilon n\rfloor)}.\end{equation}
To prove this, notice that
\[X_{(i)}^\epsilon=  \inf\{t\in\R\,:\,\#\{j\in[n]\,:\, X^\epsilon_{j}\leq t\}\geq i\}.\]
Now, if we take $t=X_{(i+\lfloor \epsilon n\rfloor)}$ above, we see that $X_{j}\leq t$ for at least $i+\lfloor \epsilon n\rfloor$ indices $j\in[n]$. Since $X_{j}=X^\epsilon_{j}$ for all but at most $\lfloor \epsilon n\rfloor$ indices $j$, we conclude:
\[\#\{j\in[n]\,:\, X^\epsilon_{j}\leq X_{(i+\lfloor \epsilon n\rfloor)}\}\geq \#\{j\in[n]\,:\, X_{j}\leq X_{(i+\lfloor \epsilon n\rfloor)}\} - \lfloor \epsilon n\rfloor\geq i.\]
Therefore, $X^\epsilon_{(i)}\leq X_{(i+\lfloor \epsilon n\rfloor)}$. This proves that the upper bound part of the claim and the lower bound part is similar.\end{proof}

\section{Trimmed population parameters and related quantities} \label{sec:boundstrimmed}

In the previous section, we showed that the trimmed mean is an average of conditionally i.i.d.\ random variables. We also defined certain parameters $\mu^{(a,b)}$, $\sigma^{(a,b)}$ and $\Delta^{(a,b)}$ of the conditional distribution of the i.i.d.\ sum. The goal of this section is to prove results about these and other parameters that appear in the analysis of the trimmed mean.

In what follows, $X_1,\dots,X_n$ is an i.i.d.\ random sample with c.d.f. $F$ and a well-defined mean $\mu$. As in \S \ref{sub:probnotation}, we write $\nu_p^p:=\Ex{|X_1-\mu|^p}$ (for $p\geq 1$) and $\sigma^2=\nu_2^2$ for the variance. We also recall the definition of $\rho_{F,p}$ from (\ref{eq:rhoFp}). Following Proposition \ref{prop:richer}, we assume without loss that there exist random variables $\{U_i\}_{i=1}^n$ that are uniform over $(0,1)$ with $F^{-1}(U_{i})=X_i$.

\subsection{Bias of the trimmed population mean} We start with the following result. 
\begin{proposition}\label{prop:controltrimmedmeanpop}Let $0<a<b<1$ and $\xi=\xi(a,b):=1-(b-a)$. Let $p>1$ be such that $\nu_p<+\infty$. Then
\[\left|\mu-\mu^{(a,b)}\right|\leq \frac{\nu_p\,\rho_{F,p}(\xi)\,\xi^{\frac{p-1}{p}}}{1-\xi}.\]\end{proposition}

\begin{proof}Notice that 
\[\mu^{(a,b)} - \mu = \frac{\int_{a}^b\,(F^{-1}(u) - \mu)\,du}{1-\xi} = -\frac{\int_{[0,1]\backslash [a,b]}\,(F^{-1}(u) - \mu)\,du}{1-\xi}\]
where in the second identity we used that $\int_0^1(F^{-1}(u)-\mu)\,du = \Ex{F^{-1}(U)-\mu}=0$.

The set $[0,1]\backslash [a,b]$ has Lebesgue measure $\xi$. H\"{o}lder's inequality gives:
\[\left|\int_{[0,1]\backslash [a,b]}\,(F^{-1}(u) - \mu)\,du\right|\leq \xi^{\frac{p-1}{p}}\,\left(\int_{[0,1]\backslash [a,b]}\,|F^{-1}(u) - \mu|^p\,du\right)^{\frac{1}{p}}.\]
Moreover,
\[\int_{[0,1]\backslash [a,b]}\,|F^{-1}(u) - \mu|^p\,du = \Ex{|X_1-\mu|^p\,Z}\]
with $Z:={\bf 1}_{[0,1]\backslash[a,b]}(U)$ satisfies $0\leq Z\leq 1$ and $\Ex{Z}=\xi,$ so that $\Ex{|X_1-\mu|^p\,Z}\leq \nu_p^p\,\rho_{F,p}(\xi)$ by the definition of $\rho_{F,p}$ (cf. Equation (\ref{eq:rhoFp})).\end{proof}



\subsection{Bounds on the trimmed population variance} Our next step is to control $\left(\sigma^{(a,b)}\right)^2$.
\begin{proposition}\label{lem:varboundmoments}The $(a,b)$-trimmed population variance satisfies
\[\sigma^{(a,b)}\leq \frac{\sigma}{1-\xi},\]
and for any $1<q\leq 2$,
\[\sigma^{(a,b)}\leq \frac{\sqrt{2}\,\nu_p^{\frac{q}{2}}\,\left(\Delta^{(a,b)}\right)^{1-\frac{q}{2}}}{1-\xi}.\]
If $\sigma<+\infty$, $\xi<1/2$ and $\rho_{F,2}(\xi)<1/3$,
\[\sigma \leq \sqrt{\frac{1-\xi}{1-\left(\frac{2-\xi}{1-\xi}\right)\,\rho^2_{F,2}(\xi)}}\,\sigma^{(a,b)}.\]

\end{proposition}
\begin{proof}The first statement follows from a simple chain of inequalities:
\begin{eqnarray*}\left(\sigma^{(a,b)}\right)^2 &=& \frac{1}{2\,(1-\xi)^2}\int_{a}^b\int_{a}^{b}\,(F^{-1}(u) - F^{-1}(v))^2\,du\,dv \\ \mbox{(integrand is $\geq 0$)}&\leq & \frac{1}{2\,(1-\xi)^2}\int_{0}^1\int_{0}^{1}\,(F^{-1}(u) - F^{-1}(v))^2\,du\,dv\\ &=& \frac{\sigma^2}{(1-\xi)^2}.\end{eqnarray*}
The second statement is similar, as
\[\forall (u,v)\in [a,b]^2\,:\, (F^{-1}(u) - F^{-1}(v))^2\leq (\Delta^{(a,b)})^{2-q}\,|F^{-1}(u) - F^{-1}(v)|^{q},\]
so that:
\[\left(\sigma^{(a,b)}\right)^2\leq \frac{(\Delta^{(a,b)})^{2-q}}{2(1-\xi)^ 2}\int_0^1\int_0^1|F^{-1}(u) - F^{-1}(v)|^{q}\,du\,dv.\]
The integral in the RHS is $\|X_1-X_2\|^q_{L^q}$ which (by convexity) is at most $2^{q}\|X-\mu\|_{L^q}^q = 2^{q}\,\nu_q^q.$ Plugging this back above, and noting that $q\leq 2$, suffices to obtain the desired inequality. 

To prove the second statement in the theorem, we use that
\[\sigma^2 = (1-\xi) \frac{1}{b-a}\int_{a}^b(F^{-1}(u)-\mu)^2\,du + \int_{[0,1]\backslash[a,b]}\,(F^{-1}(u)-\mu)^2\,du.\]
The first integral in the RHS is
\[(\sigma^{(a,b)})^2 + (\mu^{(a,b)}-\mu)^2\leq (\sigma^{(a,b)})^2 + \frac{\rho^2_{F,2}(\xi)\,\sigma^2}{(1-\xi)^2}\]
by Proposition \ref{prop:controltrimmedmeanpop}. We also have
\[ \int_{[0,1]\backslash[a,b]}\,(F^{-1}(u)-\mu)^2\,du\leq \rho^2_{F,2}(\xi)\,\sigma^2.\]
Therefore, 
\[\sigma^2 \leq (1-\xi)(\sigma^{(a,b)})^2 + \left(\frac{2-\xi}{1-\xi}\right)\,\rho^2_{F,2}(\xi)\sigma^2.\]
If $\xi<1/2$ and $\rho_{F,2}(\xi)<1/3$, then 
\[\left(\frac{2-\xi}{1-\xi}\right)\,\rho^2_{F,2}(\xi)<1\]
and
\[\sigma^2 \leq \frac{1-\xi}{1-\left(\frac{2-\xi}{1-\xi}\right)\,\rho^2_{F,2}(\xi)}\,(\sigma^{(a,b)})^2.\]
\end{proof}

\subsection{Bounds on trimmed population centered moments.} We will also need a simple proposition relating the trimmed population quantities $\nu_p^{(a,b)}$ introduced in Definition \ref{def:trimmedpop} to $\nu_p$. 

\begin{proposition}\label{prop:nuptrim}Let $0<a<b<1$ and $\xi = 1-(b-a)$. We have the bound:
\[\nu_p^{(a,b)}\leq \left(\frac{1}{(1-\xi)^{\frac{1}{p}}}+\frac{\xi^{\frac{p-1}{p}}}{(1-\xi)}\right)\nu_p.\]
\end{proposition}
\begin{proof}Notice that 
\[\nu_p^{(a,b)} = \left(\frac{1}{b-a}\int_{a}^b|F^{-1}(u) - \mu^{(a,b)}|^p\right)^{\frac{1}{p}}\leq |\mu-\mu^{(a,b)}| + \left(\frac{1}{1-\xi}\int_{a}^b|F^{-1}(u) - \mu|^p\right)^{\frac{1}{p}}.\]
Extending the range of the integral in the RHS to $[0,1]$ can only increase its value. Therefore, 
\[\nu_p^{(a,b)} \leq |\mu-\mu^{(a,b)}|+\frac{\nu_p}{(1-\xi)^{\frac{1}{p}}}.\] 
We finish the proof by applying Proposition \ref{prop:controltrimmedmeanpop}, bounding $\rho_{F,p}(\xi)\leq 1$ and performing some simple calculations.  

\end{proof}

\subsection{Order statistics of uniform random variables.} When applying the above bounds, we will take $a=U_{(k_1)}$ and $b=U_{(n-k_2+1)}$. The corresponding value of $\xi$ is the random variable $\Xi$ studied in the next Proposition. 
\begin{proposition}\label{prop:xi}Let $\Xi:= 1 - U_{(n-k_2+1)} + U_{(k_1)}$. Then for any $t>0$, 
\[\Pr{\Xi> \frac{(\sqrt{k_1+k_2-1}+\sqrt{t})^2}{n}}\leq e^{-t}.\]\end{proposition}
\begin{proof}General properties of order statistics of uniforms imply that $\Xi$ has the same law as $U_{(k_1+k_2)}$. The proof finishes via an application of Lemma \ref{lem:orderstats} in the Appendix.\end{proof}

\subsection{The trimmed width.} Finally, we present a bound on $\Delta_{n,k_1,k_2} = \Delta^{(U_{(k_1)},U_{(n-k_2+1)})}$. This is the content of the next proposition. 

\begin{proposition}\label{prop:Deltasize}Assume that $p\geq 1$ and $\nu_p^p=\Ex{|X_1-\mu|^p}<+\infty$. Let $\rho_{F,p}$ be as in Proposition \ref{prop:rhop}. Take $n,k_1,k_2\in\N$ with $k_1+k_2<n$ and set $k_{\min} = \min\{k_1,k_2\}>0$. Then for any $t>0$, 
\[\Pr{\Delta_{n,k_1,k_2}>t\,\nu_p\,\left(\frac{n}{k_{\min}}\right)^{\frac{1}{p}}}\leq \left(\frac{e\,2^p\,\rho_{F,p}\left(\frac{2^p}{t^p}\frac{k_{\min}}{n}\right)^p}{t^p}\right)^{k_{\min}}.\]
\end{proposition}
\begin{proof}
A sufficient condition for 
\[X_{(n-k_{2}+1)} - X_{(k_1)}\leq t\,\nu_p\,\left(\frac{n}{k_{\min}}\right)^{\frac{1}{p}}\]
is that 
\[\#\left\{i\in[n]\,:\,|X_{i}-\mu|>\frac{t}{2}\,\nu_p\,\left(\frac{n}{k_{\min}}\right)^{\frac{1}{p}}\right\}\leq k_{\min}-1.\]
Therefore,  
\[\Pr{\Delta_{n,k_1,k_2}>t\,\nu_p\,\left(\frac{n}{k_{\min}}\right)^{\frac{1}{p}}}\leq \Pr{\bigcup_{S\subset [n]\,:\,\#S=k_{\min}}\left(\bigcap_{i\in S}\left\{|X_i-\mu|\geq \frac{t\,\nu_p}{2}\,\left(\frac{n}{k_{\min}}\right)^{\frac{1}{p}}\right\}\right)}.\]
Using a union bound, the fact that the $X_i$ are i.i.d., a standard bound for the binomial coefficient, and Proposition \ref{prop:rhop}, we obtain:

\begin{align*}\Pr{\Delta_{n,k_1,k_2}>\frac{t\,\nu_p}{2}\,\left(\frac{n}{k_{\min}}\right)^{\frac{1}{p}}} \leq & \binom{n}{k_{\min}}\,\Pr{|X_1-\mu|\geq \frac{t\,\nu_p}{2}\,\left(\frac{n}{k_{\min}}\right)^{\frac{1}{p}}}^{k_{\min}}\\ \leq & \left(\frac{en}{k_{\min}}\right)^{k_{\min}}\,\left(\frac{2^p\,\rho_{F,p}\left(\frac{2^p}{t^p}\frac{k_{\min}}{n}\right)^p}{t^p\,\left(\frac{n}{k_{\min}}\right)}\right)^{k_{\min}},\end{align*}
from which the result follows. \end{proof}

We note the following corollary for later use. 

\begin{corollary}\label{cor:Delta}For $k=k_1=k_2<n/2$, let $v>0$ be such that
\[\frac{e}{6}\rho_{F,2}\left(\frac{1}{36v^2}\frac{k}{n}\right) \leq v.\]
Then
\[\Pr{\Delta_{n,k} >12v\,\sigma\,\sqrt{\frac{n}{k}}}\leq e^{-k}.\]\end{corollary}
\begin{proof}Proposition \ref{prop:Deltasize} (with $k_{\min}=k$, $p=2$ and $t=12v$) and our assumption on $v$ imply 
\[\Pr{\Delta_{n,k} >12v\,\sigma\,\sqrt{\frac{n}{k}}}\leq \left(\frac{e}{36v^2}\rho_{F,2}\left(\frac{1}{36v^2}\frac{k}{n}\right)^2\right)^{k} = \left(\frac{e}{6v}\rho_{F,2}\left(\frac{1}{36v^2}\frac{k}{n}\right)\right)^{2k}\,e^{-k}\leq e^{-k}.\]\end{proof}

\section{Sub-Gaussian concentration}\label{sec:subGaussian}

Having laid down the groundwork in previous sections, we now proceed to investigate the behavior of the trimmed mean in the finite-variance setting. Our main goal is to prove the three theorems stated in \S \ref{sub:intro:subgaussian}.  

\subsection{A master theorem} 

As it turns out, all sub-Gaussian results follow from the following theorem (proven subsequently). 

\begin{theorem}\label{thm:generalfinitevariance}Let $X_1,\dots,X_n$ be i.i.d.\ with c.d.f. $F$, a well-defined mean $\mu=\Ex{X_1}$ and finite variance $\sigma^2\in (0,+\infty)$. Let $x>0$ be given and consider a trimming parameter  $\lceil x^2/2\rceil \leq k$. Assume
\[\xi_*:=(\sqrt{2}+1)^2\,\frac{k}{n}<1.\]
and let $v>0$ be such that  
\begin{eqnarray}\label{eq:assumptionsubgaussiants1}\frac{e}{6}\rho_{F,2}\left(\frac{1}{36v^2}\frac{k}{n}\right) &\leq& v; \\ \label{eq:assumptionsubgaussiants2}(\sqrt{2}+1)\,\rho_{F,2}\left(\xi_*\right)& \leq&  v.\end{eqnarray}
Then
\begin{equation}\label{eq:generalfinitevariance}\Pr{|\overline{X}_{n,k}-\mu|>\sigma \,\frac{x+ h(\xi_*,v)\sqrt{2k}}{\sqrt{n}}}\leq 2\,\exp\left(-\frac{x^2}{2}\right) + 2\exp\left(-k\right)\end{equation}
\[\mbox{ where }h(\xi_*,v) := (1-\xi_*)^{-3/2}-1+\frac{\sqrt{2}\,v}{1-\xi_*}.\]\end{theorem}

\begin{proof}We work under the framework of Proposition \ref{prop:richer}, whereby we may assume that $X_i=F^{-1}(U_i)$ and $X_{(i)}=F^{-1}(U_{(i)})$ for i.i.d.\ random variables $\{U_i\}_{i\in [n]}$ that are uniform over $(0,1)$. We write $\Xi:=1-(U_{(n-k+1)}-U_{(k)})$ and recall that $\Delta_{n,k} = X_{(n-k+1)}-X_{(k)} = F^{-1}(U_{(n-k+1)})-F^{-1}(U_{(k)})=\Delta^{(U_{(k)},U_{(n-k+1)})}$.

Corollary \ref{cor:conditionalsample} implies that Bernstein's inequality (Theorem \ref{thm:bernstein}) applies conditionally on $U_{(k)},U_{(n-k+1)}$:  $\overline{X}_{n,k}$ is an average of $n-2k$ i.i.d.\ random variables with mean $\mu^{(U_{(k)},U_{(n-k+1)})}$ and variance $(\sigma^{(U_{(k)},U_{(n-k+1)})})^2$. Moreover, the random variables in the average, when centered, are bounded by $\Delta_{n,k}$ in absolute value. We obtain:
\begin{equation}\label{eq:bernsteinsubgaussian}\Pr{\left.|\overline{X}_{n,k}-\mu^{(U_{(k)},U_{(n-k+1)})}|>\frac{x\sigma^{(U_{(k)},U_{(n-k+1)})}}{\sqrt{n-2k}} + \frac{x^2\Delta_{n,k}}{12(n-2k)}\right|(U_{(k)},U_{(n-k+1)})}\leq 2e^{-\frac{x^2}{2}}.\end{equation}

Our next step is to define a ``good event'' where the the trimmed population parameters $\mu^{(U_{(k)},U_{(n-k+1)})}$, $\sigma^{(U_{(k)},U_{(n-k+1)})}$ and $\Delta_{n,k}$ satisfy deterministic bounds. Specifically, let ${\sf Good}$ denote the event where the following two inequalities hold. 
\begin{eqnarray}\label{eq:GoodXiSubGaussian}\Xi &\leq & \xi_*:=\frac{(\sqrt{2}+1)^2\,k}{n};\\ \label{eq:GoodDeltaSubGaussian}
\Delta_{n,k} &\leq & 12v\,\sigma\,\sqrt{\frac{n}{k}}.
\end{eqnarray}
When {\sf Good} holds, Proposition \ref{prop:controltrimmedmeanpop} and assumption (\ref{eq:assumptionsubgaussiants2}) give
\[|\mu-\mu^{(U_{(k)},U_{(n-k+1)})}|\leq \frac{\sigma\,\rho_{F,2}(\xi_*)\,\sqrt{\xi_*}}{1-\xi_*}\leq \frac{v}{1-\xi_*}\sigma\,\sqrt{\frac{k}{n}}.\]
Moreover, Proposition \ref{lem:varboundmoments} and the inequalities $n-2k\geq (1-\xi_*)\,n$, $k\geq x^2/2$ imply
\[\frac{x\sigma^{(U_{(k)},U_{(n-k+1)})}}{\sqrt{n-2k}}\leq \frac{x\sigma}{(1-\xi_*)^{3/2}\sqrt{n}}\leq \frac{\sigma}{\sqrt{n}}\,(x +  ((1-\xi_*)^{-3/2}-1)\sqrt{2k})\]
and (\ref{eq:GoodDeltaSubGaussian}) combined with $k\geq \lceil x^2/2\rceil$ gives:
\[\frac{x^2\Delta_{n,k}}{12(n-2k)}\leq \frac{2v}{1-\xi_*}\,\sigma\,\sqrt{\frac{k}{n}}.\]

The upshot of this discussion is that, when {\sf Good} holds,
\[|\mu-\mu^{(U_{(k)},U_{(n-k+1)})}| + \frac{x\sigma^{(U_{(k)},U_{(n-k+1)})}}{\sqrt{n-2k}} + \frac{x^2\Delta_{n,k}}{12(n-2k)}\leq \frac{x\sigma}{\sqrt{n}} + h(\xi_*,v)\,\sigma\,\sqrt{\frac{2k}{n}}.\]

and we obtain from (\ref{eq:bernsteinsubgaussian}) that 
\begin{equation}\Pr{\left\{|\overline{X}_{n,k}-\mu|>\frac{x\sigma}{\sqrt{n}} +h(\xi_*,v)\sigma\,\sqrt{\frac{2k}{n}}\right\}\cap {\sf Good}}\leq 2e^{-\frac{x^2}{2}}.\end{equation}

To finish the proof,we show that $\Pr{{\sf Good}^c}\leq 2e^{-k}$. Since ${\sf Good}^c$ is the event where either (\ref{eq:GoodXiSubGaussian}) or (\ref{eq:GoodDeltaSubGaussian}) do not hold, it suffices to bound the corresponding probabilities individually by $e^{-k}$.

Proposition \ref{prop:xi} (applied with $k_1=k_2=t=k$) give \[\Pr{\mbox{(\ref{eq:GoodXiSubGaussian}) does not hold}}\leq e^{-k},\]
whereas Corollary \ref{cor:Delta} and assumption (\ref{eq:assumptionsubgaussiants1}) give
\[\Pr{\mbox{(\ref{eq:GoodDeltaSubGaussian}) does not hold}}\leq e^{-k}.\]

\end{proof}

\subsection{Proofs of the main sub-Gaussian results}

We now use Theorem \ref{thm:generalfinitevariance} to obtain the three theorems in \S \ref{sub:intro:subgaussian}.
\subsubsection{Proof of Theorem \ref{thm:allsubgaussian}}\label{subsub:proof:allsubgaussian}
\begin{proof} We will apply Theorem \ref{thm:generalfinitevariance} with $2k=2\lceil x^2/2\rceil\leq 2+x^2$. With this choice, the theorem gives 
\begin{equation}\label{eq:allvariance}\Pr{|\overline{X}_{n,k}-\mu|>\sigma \,\frac{\sqrt{2}h(\xi_*,v)+ (1+h(\xi_*,v))\,x)}{\sqrt{n}}}\leq 4\,\exp\left(-\frac{x^2}{2}\right)\end{equation}
whenever $v$ satisfies conditions (\ref{eq:assumptionsubgaussiants1}) and (\ref{eq:assumptionsubgaussiants2}). Now, our assumptions imply $\xi_*\leq 1/2$, so that $(1-\xi_*)^{-3/2}\leq 2\sqrt{2}$. Using the bounds $\rho_{F,2}\leq 1$ and $e\leq 6$, we see that $v:=\sqrt{2} + 1$ satisfies (\ref{eq:assumptionsubgaussiants1}) and (\ref{eq:assumptionsubgaussiants2}), and
\[h(\xi_*,v)\leq 3+4\sqrt{2}.\]
Plugging this back into (\ref{eq:allvariance}) finishes the proof.\end{proof}

\subsubsection{Proof of Theorem \ref{thm:sharpersubgaussian}}\label{subsub:proof:sharpersubgaussian}
\begin{proof}Like with the previous proof, we apply Theorem \ref{thm:generalfinitevariance}. With the choice $k=k(x)=\lceil x^2/2\rceil$, we see from (\ref{eq:allvariance}) that we can obtain the desired probability bound if we can show $h(\xi_*,v)\leq a$ for some valid choice of $\xi_*,v$.

The first case of the theorem can be dealt with as follows. Fix $v:=a/8$ and notice that $h(\xi_*,v)\leq a$ for small enough $\xi_*=(\sqrt{2}+1)^2k/n$. Moreover, conditions
(\ref{eq:assumptionsubgaussiants1}) and (\ref{eq:assumptionsubgaussiants2}) correspond to
\begin{eqnarray*}\label{eq:assumptionsubgaussiantsXXX}\frac{e}{6}\rho_{F,2}\left(\frac{2}{a^2}\frac{k}{n}\right) &\leq& \frac{a}{8}; \\ \label{eq:assumptionsubgaussiantsYYY}(\sqrt{2}+1)\,\rho_{F,2}\left((\sqrt{2}+1)^2\,\frac{k}{n}\right)& \leq&  \frac{a}{8};\end{eqnarray*}
and are automatically satisfied if $k/n$ is small enough, since $\rho_{F,2}(\xi)\to 0$ as $\xi\to 0$. Therefore, Theorem \ref{thm:generalfinitevariance} can be applied whenever $k/n\leq f_F(a)$ for some value $f_F(a)>0$ depending only on $F$ and $a$. Recalling the fact that $\sqrt{k/n}\leq (1+x)/\sqrt{n}$, we see that it suffices to require $(1+x)/\sqrt{n}\leq \eta_F(a)$ where $\eta_F(a):=\sqrt{f_F(a)}$.

are satisfied for $v=a/8$ and $\xi_*$ small enough (depending solely on $\rho_{F,2}$). 

For the second case, we assume $\nu_p^p=\Ex{|X_1-\mu|^p}\leq (\kappa_{2,p}\sigma)^p$ for some $p>2$. We aim at selecting the a value $v_*$ satisfying (\ref{eq:assumptionsubgaussiants1}) and (\ref{eq:assumptionsubgaussiants2}), and then plug this value into $h$ to obtain a bound. 

 To this end, recall that Proposition \ref{prop:rhop} gives the following bound on $\rho_{F,2}$:
\[\forall \xi\in [0,1],\, \rho_{F,2}(\xi)\leq \kappa_{2,p}\,\xi^{\frac{1}{2}-\frac{1}{p}}\] 
and the quantities depending on $\rho_{F,2}$ in conditions (\ref{eq:assumptionsubgaussiants1}) and (\ref{eq:assumptionsubgaussiants2}) can be bounded as follows:
\begin{eqnarray}\label{eq:conditionsharper1}
    \frac{e}{6}\rho_{F,2}\left(\frac{1}{36v^2}\frac{k}{n}\right) &\leq & \frac{\kappa_{2,p}}{v^{\frac{p-2}{p}}}\,\left(\frac{k}{n}\right)^{\frac{p-2}{2p}};\\ \label{eq:conditionsharper2}
   (\sqrt{2}+1)\,\rho_{F,2}\left(\frac{(\sqrt{2}+1)^2\,k}{n}\right)& \leq&  \kappa_{2,p}\,6^{\frac{p-1}{p}}\,\left(\frac{k}{n}\right)^{\frac{p-2}{2p}}.
\end{eqnarray}
Notice that we used $e\leq 6$ and $(\sqrt{2}+1)^2\leq 6$ and omitted a few constants $\leq 1$ to simplify the calculations above. 

A sufficient condition for $v$ to satisfy (\ref{eq:assumptionsubgaussiants1}) and (\ref{eq:assumptionsubgaussiants2}) is that it is an upper bound on the RHS of (\ref{eq:conditionsharper1}) and (\ref{eq:conditionsharper2}). That is, we need that
\begin{equation}\label{eq:choiceofvsharper}v\geq v_0:=\max\left\{{\kappa_{2,p}}^{\frac{2}{2p-2}}\,\left(\frac{k}{n}\right)^{\frac{p-2}{4p-4}},\kappa_{2,p}\,6^{\frac{p-1}{p}}\,\left(\frac{k}{n}\right)^{\frac{p-2}{2p}}\right\}.\end{equation}
To obtain a cleaner value, we note that $\kappa_{2,p}\geq 1$ (by Jensen's inequality) and that, since $k/n\leq 1/2$ by assumption, 
\[v_0\leq v_*:= 6\kappa_{2,p}\left(\frac{k}{n}\right)^{\frac{p-2}{4p-4}},\]
which is the value we use in what follows. Let us record it for later use. 
\begin{proposition}\label{prop:valueofv_*} If $\nu_p\leq \kappa_{2,p}\sigma$ for some $p>2$, the choice of \[v=v_*:= 6\kappa_{2,p}\left(\frac{k}{n}\right)^{\frac{p-2}{4p-4}}\]
satisfies (\ref{eq:assumptionsubgaussiants1}) and (\ref{eq:assumptionsubgaussiants2}).\end{proposition}

To finish, notice that, if $\xi_*\leq 1/2$,
\[(1-\xi_*)^{-3/2}-1\leq \frac{3\xi_*}{2\,(1-\xi_*)^{5/2}}\leq 6\sqrt{2}\xi_*\]
and 
\[h(\xi_*,v)\leq 6\sqrt{2}\,\xi_*+4v \leq  216\sqrt{2}\,\frac{k}{n} + 4v.\]

For $v=v_*$ and $k=\lceil x^2/2\rceil\leq (1+x)^2$, the RHS above is the quantity appearing in the RHS of (\ref{eq:conditiononasharper}) in the statement of the Theorem. This ensures $h(\xi_*,v_*)\leq a$.\end{proof}

\subsubsection{Proof of Theorem \ref{thm:multiplesubgaussian}}\label{subsub:proof:multiplesubgaussian}

\begin{proof}We will apply Theorem \ref{thm:generalfinitevariance} with a choice of $v$ that guarantees $h(\xi_*,v)\leq 1/\sqrt{2k_*}$; notice that this will immediately lead to the desired bound. 

For this purpose, we can reuse the calculations starting around (\ref{eq:conditionsharper1}) and (\ref{eq:conditionsharper2}) in the previous proof, with $k_*$ replacing $k$, and obtain that, for $v_*$ as in Proposition \ref{prop:valueofv_*}, we have $h(\xi_*,v_*)$:
\[h(\xi_*,v_*) \leq  216\sqrt{2}\,\frac{k_*}{n} + 24\kappa_{2,p}\left(\frac{k_*}{n}\right)^{\frac{p-2}{4p-4}}.\]

To finish, we note that the assumption on $k_*$ in the theorem ensures that $h(\xi_*,v_*)\leq 1/\sqrt{2k_*}$, as desired.\end{proof}

\section{Precise Gaussian approximation and confidence intervals}\label{sec:CLT}

In this section, we prove the results stated in \S \ref{sub:intro:CLT}. We first show that trimmed mean satisfies the same Central Limit Theorem under a finite variance condition. Under stronger assumptions, we show that the trimmed mean is nearly Gaussian even when one goes very deeply into its left and right tails. Finally, this stronger result will be shown to have strong implications for constructing confidence intervals.  

\subsection{A general result}

The first part of this section presents general Gaussian approximation result from which the main theorems in this section will follow. This is analogous to what we did with Theorem \ref{thm:generalfinitevariance} in Section \ref{sec:subGaussian}. 

\begin{theorem}\label{thm:generalpreciseGaussian}There exists a universal constant $C>0$ such that the following holds. Let $X_1,\dots,X_n$ be i.i.d.\ with c.d.f. $F$, a well-defined mean $\mu=\Ex{X_1}$ and finite variance $\sigma^2\in (0,+\infty)$. Let $x>0$ be given and consider a trimming parameter $k\in\N$ with $k\geq 2$. Assume
\[\xi_*:=(\sqrt{2}+1)^2\,\frac{k}{n}\leq \frac{1}{2},\]
and let $v>0$ be such that  
\begin{eqnarray}\label{eq:conditionv1preciseGaussian}
\frac{e}{6}\rho_{F,2}\left(\frac{1}{36v^2}\frac{k}{n}\right) &\leq& v;\\ \label{eq:conditionv2preciseGaussian}
\rho_{F,2}\left(\xi_*\right)&\leq & v;\\
\label{eq:conditionxpreciseGaussian}(1+x)\,v\,\sqrt{k} &\leq & \frac{1}{C};\\
\label{eq:conditionkpreciseGaussian}(1+x)\,\frac{k}{n} &\leq & \frac{1}{C};\\ \label{eq:conditionx3preciseGaussian}
\frac{(1+x)^3\,v}{\sqrt{k}}&\leq & \frac{1}{C}.
\end{eqnarray}
Then:
\begin{equation}\label{eq:1ststatementGaussian}\left|\Pr{\overline{X}_{n,k}>\mu+\frac{x\sigma}{\sqrt{n}}} - (1-\Phi(x))\right|\leq \eta\,(1-\Phi(x)) + 4e^{-k},\end{equation}
where 
\[\eta:=C\left(\frac{(1+x)^3v}{\sqrt{k}} + (1+x)^2v\sqrt{k} + \frac{(1+x)^2\,k}{n}\right).\]
If in addition $1\leq x<\sqrt{n-2k}/2$, we also have
\[\left|\Pr{\overline{X}_{n,k}>\mu+\frac{x\hat{\sigma}_{n,k}}{\sqrt{n}}} - (1-\Phi(x))\right|\leq\eta\,(1-\Phi(x)) + 4e^{-k}\]
and
\[\Pr{\hat{\sigma}_{n,k}>C\,(1+ (v\sqrt{k}+k/n))\,\sigma}\leq 4e^{-k}.\]
\end{theorem}

\begin{proof} Once again, it will be useful to invoke Proposition \ref{prop:richer} and take $X_i=F^{-1}(U_i)$ and $X_{(i)}=F^{-1}(U_{(i)})$ for i.i.d.\ random variables $\{U_i\}_{i\in [n]}$ that are uniform over $(0,1)$. Recall $\Delta_{n,k} = X_{(n-k+1)}-X_{(k)} = F^{-1}(U_{(n-k+1)})-F^{-1}(U_{(k)})$.\\

\noindent\underline{Proof outline.} The proof will consist of four steps. In the first one, we prove a conditional Gaussian tail approximation given $(U_{(k)},U_{(n-k+1)})$. This approximation will have the form
\begin{equation}\label{eq:desiredboundCLTgeneral}\left|\Pr{\left.\overline{X}_{n,k}>\mu+\frac{x\sigma}{\sqrt{n}}\right| U_{(k)},U_{(n-k+1)}} - (1-\Phi(x))\right|\leq R_{n}(x)\,(1-\Phi(x))\, + I_n(x)\end{equation}
where $R_{n}(x)$ and $I_n(x)$ are deterministic functions of $(x,U_{(k)},U_{(n-k+1)})$. 

For the second step, we introduce an event ${\sf Good}$  and show that $I_n(x)=0$ and $R_n(x)\leq \eta$ for a suitable deterministic value $\eta$. As a result, 
\[\left|\Pr{\left.\overline{X}_{n,k}>\mu+\frac{x\sigma}{\sqrt{n}}\right| U_{(k)},U_{(n-k+1)}} - (1-\Phi(x))\right|\,{\bf 1}_{\sf Good}\leq \eta\,(1-\Phi(x)),\]
and so
\begin{align}\nonumber
    & \,\left|\Pr{\overline{X}_{n,k}>\mu+\frac{x\sigma}{\sqrt{n}}} - (1-\Phi(x))\right| \\ \nonumber\leq &\,\Ex{\left|\Pr{\left.\overline{X}_{n,k}>\mu+\frac{x\sigma}{\sqrt{n}}\right| U_{(k)},U_{(n-k+1)}} - (1-\Phi(x))\right|} \\ \nonumber\leq &\,  \Ex{\left(\left|\Pr{\left.\overline{X}_{n,k}>\mu+\frac{x\sigma}{\sqrt{n}}\right| U_{(k)},U_{(n-k+1)}} - (1-\Phi(x))\right|{\bf 1}_{\sf Good}\right) + {\bf 1}_{\sf Good^c}}\\ 
    \label{eq:desiredboundinfullCLTgeneral}
    \leq & \,\eta\,(1-\Phi(x)) + \Pr{{\sf Good}^c}, 
\end{align}
In Step 3, we bound the probability of ${\sf Good}^c$ and finish the proof of (\ref{eq:1ststatementGaussian}). The fourth and final step of the proof adapts the above argument to finish the proof of Theorem \ref{thm:generalpreciseGaussian}.\\ 

\noindent\underline{Step 1: conditional Gaussian approximation.} Write
\[V_{n,k}^2 := \frac{1}{n-2k}\sum_{i=k+1}^{n-k}(X_{(i)}-\mu^{(U_{(k)},U_{(n-k+1)})})^2 = \hat{\sigma}^{2}_{n,k} + (\overline{X}_{n,k}-\mu^{(U_{(k)},U_{(n-k+1)})})^2.\]
The CLT from \cite{Jing2003}, in the form given by Theorem \ref{thm:CLTJing} above, may be combined with Corollary \ref{cor:conditionalsample}, which implies that the random variables $\{X_{(i)}\}_{i=k+1}^{n-k}$ are conditionally i.i.d.\ (up to their ordering). We obtain that for any $x>0$,
\begin{align} \nonumber&\left|\Pr{\left.\overline{X}_{n,k}>\mu^{(U_{(k)},U_{(n-k+1)})}+\frac{xV_{n,k}}{\sqrt{n-2k}}\right| U_{(k)},U_{(n-k+1)}} - (1-\Phi(x))\right|\\ \label{eq:conditionalCLT} \leq & \frac{A\Delta_{n,k}(1+x)^3}{\sigma^{(U_{(k)},U_{(n-k+1)})}\sqrt{n-2k}}\,(1-\Phi(x)) + {\bf 1}_{\left\{\frac{\Delta_{n,k}(1+x)^3}{\sigma^{(U_{(k)},U_{(n-k+1)})}\sqrt{n-2k}}>\frac{1}{A}\text{ or }\sigma^{(U_{(k)},U_{(n-k+1)})}=0\right\}}.\end{align}
where $A$ is universal. Now define, for a given $x>0$,
\begin{equation}\label{eq:defHnk}H_{n,k}(x):=\frac{\sqrt{n}}{\sigma}\,\left(\mu^{(U_{(k)},U_{(n-k+1)})} - \mu + \frac{xV_{n,k}}{\sqrt{n-2k}} - \frac{x\sigma}{\sqrt{n}}\right).\end{equation}
Equation (\ref{eq:conditionalCLT}) can be rewritten as:
\begin{align*} \nonumber&\left|\Pr{\left.\overline{X}_{n,k}>\mu+\frac{(x+H_{n,k}(x))\,\sigma}{\sqrt{n}} \right| U_{(k)},U_{(n-k+1)}} - (1-\Phi(x))\right|\\ \leq & \frac{A\Delta_{n,k}(1+x)^3}{\sigma^{(U_{(k)},U_{(n-k+1)})}\sqrt{n-2k}}\,(1-\Phi(x)) + {\bf 1}_{\left\{\frac{\Delta_{n,k}(1+x)^3}{\sigma^{(U_{(k)},U_{(n-k+1)})}\sqrt{n-2k}}>\frac{1}{A}\text{ or }\sigma^{(U_{(k)},U_{(n-k+1)})}=0\right\}}.\end{align*}
Now, the above is an inequality for a conditional probability given $U_{(k)},U_{(n-k+1)}$. For any given $x\geq 0$, the random variable $H_{n,k}(x)$ is a function of $U_{(k)},U_{(n-k+1)}$; therefore, we can apply the preceding with $x+H_{n,k}(x)$ replacing $x$ and obtain that, almost surely, 
\begin{align*} \nonumber&\left|\Pr{\left.\overline{X}_{n,k}>\mu+\frac{x\,\sigma}{\sqrt{n}} \right| U_{(k)},U_{(n-k+1)}} - (1-\Phi(x-H_{n,k}(x)))\right|\\ \leq & \frac{A\Delta_{n,k}(1+|x-H_{n,k}(x)|)^3}{\sigma^{(U_{(k)},U_{(n-k+1)})}\sqrt{n-2k}}\,(1-\Phi(x-H_{n,k}(x))) + {\bf 1}_{\left\{\frac{\Delta_{n,k}(1+|x-H_{n,k}(x)|)^3}{\sigma^{(U_{(k)},U_{(n-k+1)})}\sqrt{n-2k}}>\frac{1}{A}\text{ or }\sigma^{(U_{(k)},U_{(n-k+1)})}=0\right\}},\end{align*}

To finish this chain of inequalities, note that, if $3|H_{n,k}(x)|\max\{x,1\}\leq 1$, then the Gaussian tail perturbation bound in Proposition \ref{prop:tailgaussian} of the Appendix gives:
\[e^{-3|H_{n,k}(x)|\max\{x,1\}}\leq \frac{1-\Phi(x-H_{n,k}(x))}{1-\Phi(x)}\leq e^{3|H_{n,k}(x)|\max\{x,1\}},\]
and thus 
\[\frac{|1-\Phi(x-H_{n,k}(x))) - (1-\Phi(x))|}{1-\Phi(x)}\leq e^{3|H_{n,k}(x)|\max\{x,1\}}-1\leq 3\,(e-1)\,|H_{n,k}(x)|\max\{x,1\}.\]
Therefore, 
\[\left|\Pr{\left.\overline{X}_{n,k}>\mu+\frac{x\sigma}{\sqrt{n}}\right| U_{(k)},U_{(n-k+1)}} - (1-\Phi(x))\right|\leq  R_n(x)\,(1-\Phi(x)) + I_n(x)\]
where 
\begin{eqnarray}\label{eq:defRn(x)}R_n(x)&:=&\frac{A\Delta_{n,k}(1+|x-H_{n,k}(x)|)^3}{\sigma^{(U_{(k)},U_{(n-k+1)})}\sqrt{n-2k}}+3\,(e-1)\,|H_{n,k}(x)|\max\{x,1\} \mbox{ and }\\ 
\label{eq:defIn(x)} I_n(x)&:=& {\bf 1}_{\left\{R_n(x)>1\text{ or }\sigma^{(U_{(k)},U_{(n-k+1)})}=0\right\}}\end{eqnarray}
are both deterministic functions of $(x,U_{(k)},U_{(n-k+1)})$. This is a bound of the form discussed in the proof outline; see (\ref{eq:desiredboundCLTgeneral}).\\ 
 
\noindent\underline{Step 2: The good event.} Define $\Xi:=1-U_{(n-k+1)}+U_{(k)}$. Let {\sf Good} be the event where the following three inequalities hold. 
\begin{eqnarray}\label{eq:conditionXiGoodCLT}\Xi&\leq & \xi_*:=(\sqrt{2}+1)^2\,\frac{k}{n};\\ \label{eq:conditionDeltaGoodCLT}\Delta_{n,k}&\leq & 12v\,\sigma\,\sqrt{\frac{n}{k}};\\ \label{eq:conditionVGoodCLT}
\left|V_{n,k}-\sigma^{(U_{(k)},U_{(n-k+1)})}\right| &\leq & \frac{k\Delta_{n,k}}{\sqrt{n-2k}}+ \frac{k^2\,\Delta^2_{n,k}}{12n\,\sigma^{(U_{(k)},U_{(n-k+1)})}}. 
\end{eqnarray}


Our next goal is to show that the occurrence of this event will allow us to control the random variables $H_{n,k}(x)$, $R_n(x)$ and $I_n(x)$ that appeared in Step 1. In doing this, we will use $C_0$ to denote a universal constant whose value may change from line to line. We will also assume (as we may) that the constant $C>0$ in our assumptions is suitably large (in particular, we assume $C>4$ for the first inequality below). 

For the remainder of this step, assume {\sf Good} holds. We may deduce from Propositions \ref{prop:controltrimmedmeanpop} and \ref{lem:varboundmoments} that
\[\frac{|\mu-\mu^{(U_{(k)},U_{(n-k+1)})}|}{\sigma}\leq \frac{\rho_{F,2}(\xi_*)\sqrt{\xi_*}}{1-\xi_*}\leq C_0\,v\,\sqrt{\frac{k}{n}}\leq \frac{1}{C}<1/4.\]
By the same kind of reasoning, \[\left|\frac{\sigma^{(U_{(k)},U_{(n-k+1)})}}{\sigma}-1\right| \leq \max\left\{\left|\frac{1}{1-\xi_*}-1\right|, \left|1 - \sqrt{\frac{1-\left(\frac{2-\xi_*}{1-\xi_*}\right)\,\rho^2_{F,2}(\xi_*)}{1-\xi_*}}\right|\right\} \leq C_0\,\left(v + \frac{k}{n}\right)<1/4.\]
In particular, the above implies that $\sigma^{(U_{(k)},U_{(n-k+1)})}\geq 3\sigma/4>0$. From (\ref{eq:conditionDeltaGoodCLT}) and (\ref{eq:conditionVGoodCLT}) and our assumption that $v\sqrt{k}\leq 1/(1+x)C$, we also obtain:
\[\frac{\left|V_{n,k} - \sigma^{(U_{(k)},U_{(n-k+1)})}\right|}{\sigma}\leq C_0\,(v\sqrt{k}).\]

The upshot of this discussion is that the quantity $H_{n,k}(x)$ defined in (\ref{eq:defHnk}) satisfies
\begin{align*}|H_{n,k}(x)| & \leq \left|\frac{\sqrt{n}(\mu^{(U_{(k)},U_{(n-k+1)})}-\mu)}{\sigma}\right|  + \left| \frac{x\sqrt{n}\,(V_{n,k}-\sigma^{(U_{(k)},U_{(n-k+1)})})}{\sigma\sqrt{n-2k}}\right| +\left| \frac{x\sigma^{(U_{(k)},U_{(n-k+1)})}}{\sigma\sqrt{1-2k/n}} -x\right|\\
&\leq C_0\,\left((1+x)\,v\sqrt{k} + \frac{x\,k}{n}\right).
\end{align*}
Now, 
\[\max\{x,1\}\,C_0\,\left((1+x)\,v\sqrt{k} + \frac{x\,k}{n}\right)\leq C_0\left((1+x)^2v\sqrt{k} + \frac{(1+x)^2\,k}{n}\right).\]
Therefore, when {\sf Good} holds, $|H_{n,k}(x)|\leq C_0(1+x)$ and 
\begin{align}\nonumber R_n(x) & = \frac{A\Delta_{n,k}(1+|x-H_{n,k}(x)|)^3}{\sigma^{(U_{(k)},U_{(n-k+1)})}\sqrt{n-2k}}+3\,(e-1)\,|H_{n,k}(x)|\max\{x,1\} \\ \label{eq:defeta} & \leq \eta:=C_0\left(\frac{(1+x)^3v}{\sqrt{k}} + (1+x)^2v\sqrt{k} + \frac{(1+x)^2\,k}{n}\right).\end{align}
By our assumptions on $v,k,x$ and $n$ -- and recalling that we assume $C>0$ to be sufficiently large --, we obtain $\eta\leq 1$ under {\sf Good}. Therefore, when {\sf Good} holds,
\[R_n(x)\leq \eta\mbox{ and }I_n(x)=0,\]
as desired.

\noindent\underline{Step 3: Proof of (\ref{eq:1ststatementGaussian}).} Using (\ref{eq:desiredboundinfullCLTgeneral}), we may finish by showing $\Pr{{\sf Good}^c}\leq 3e^{-k}$.
To do this, we note that ${\sf Good}^c$ is the event that one of the inequalities (\ref{eq:conditionXiGoodCLT}), (\ref{eq:conditionDeltaGoodCLT}) and (\ref{eq:conditionVGoodCLT}) does not hold. Moreover,
\begin{eqnarray*}\Pr{\mbox{(\ref{eq:conditionXiGoodCLT}) does not hold}}&\leq & e^{-k}\mbox{ by Proposition \ref{prop:xi}};\\ \Pr{\mbox{(\ref{eq:conditionDeltaGoodCLT}) does not hold}}&\leq & e^{-k}\mbox{ by Corollary \ref{cor:Delta} and assumption (\ref{eq:conditionv1preciseGaussian})}; \mbox{ and} \\
\Pr{\mbox{(\ref{eq:conditionVGoodCLT}) does not hold}} &\leq & 2e^{-\frac{k^2}{2}}\mbox{ by Theorem \ref{thm:bernsteinvariance}},
\end{eqnarray*}
and the fact that $k\geq 2$ guarantees the desired bound. \\

\noindent\underline{Step 4: The full proof.} We now sketch the final argument for the proof of the bounds involving the trimmed standard deviation $\hat{\sigma}_{n,k}$. Going back to Step 1, notice that
\[\hat{\sigma}^2_{n,k} = V_{n,k}^2 - (\overline{X}_{n,k}-\mu^{(U_{(k)},U_{(n-k+1)})})^2.\]
The standard Bernstein inequality can be applied to obtain that
\[\Pr{\left.|\overline{X}_{n,k}-\mu^{(U_{(k)},U_{(n-k+1)})}|\leq \sqrt{\frac{2k}{n}}\,\sigma^{(U_{(k)},U_{(n-k+1)})} + \frac{k\Delta_{n,k}}{n}\right| U_{(k)},U_{(n-k+1)}}\geq 1- e^{-k}.\]
In particular, with probability $1-4e^{-k}$, both the above event and {\sf Good} hold simultaneously. Call this new event ${\sf Good}_+$. It is a simple exercise (that we omit) to argue that, when ${\sf Good}_+$ holds, both $\hat{\sigma}_{n,k}/\sigma$ and $V_{n,k}/\sigma$ are sufficiently close to $1$ that the same perturbation arguments we applied previously are still valid with $\hat{\sigma}_{n,k}$ replacing $\sigma$ in all calculations. \end{proof}

\subsection{Proof of Theorem \ref{thm:preciseconfidence}}\label{sub:proof:preciseconfidence}

\begin{proof}We want to obtain this result as a consequence of Theorem \ref{thm:generalpreciseGaussian} above. This will require that we show that our choices of $k,v$ satisfy the conditions (\ref{eq:conditionv1preciseGaussian}) to (\ref{eq:conditionx3preciseGaussian}) of that theorem. 

Our choice of $k=k_*\geq \lceil \log(4/\delta\,(1-\Phi(x)))\rceil$ guarantees $4e^{-k_*}\leq \delta\,(1-\Phi(x))$ for large $n$. We will take $v=v_*=6\kappa_{2,p}(k_*/n)^{(p-2)/(4p-4)}$ as in Proposition \ref{prop:valueofv_*} with this choice of $k_*$, which guarantees that inequalites (\ref{eq:conditionsharper1}) and (\ref{eq:conditionsharper2}) are satisfied. Notice that these two inequalities are precisely the same as (\ref{eq:conditionv1preciseGaussian}) and (\ref{eq:conditionv2preciseGaussian}) in Theorem \ref{thm:generalpreciseGaussian}. 

To check the other three conditions, first notice that, when $x>0$, $k_*\geq c + x^2/2$ for some universal $c\in\R$. Therefore, $(1+x)\leq L\,\sqrt{k_*}$ for some universal $L>0$. This shows that, in order to satisfy (\ref{eq:conditionxpreciseGaussian}), it suffices to require that $v_*\,k_*\leq 1/B$ for some universal $B>0$, which is implied by our assumption \[\gamma= \kappa_{2,p}\,\frac{k_*^{\frac{7p-8}{4p-4}}}{n^{\frac{p-2}{4p-4}}}<\frac{1}{C}.\] 
Note that this assumption also implies that 
\[\left(\frac{k_*^2}{n}\right)^{\frac{p-2}{4p-4}}\leq \frac{k_*^{\frac{3}{2}+\frac{p-2}{4p-4}}}{n^{\frac{p-2}{4p-4}}}\leq \gamma<\frac{1}{C},\]
so that 
\[\frac{k_*^2}{n}\leq \frac{\gamma^{\frac{4p-4}{p-2}}}{C^{\frac{4p-4}{p-2}}}\leq \gamma<\frac{1}{C}\]
as $4p-4>p-2$ always (here we have assumed, as we may, that the universal constant $C$ is at least $1$).
 
This can be used to check the other conditions of Theorem \ref{thm:generalpreciseGaussian}, up to further adjustments in $C$, since $(1+x)k_*/n\leq Lk_*^{3/2}/n$  and $(1+x)^3v_*/\sqrt{k_*}\leq L^3v_*k_*\leq L'\gamma$ (with $L,L'>0$ universal in both cases).

To finish the proof, we apply Theorem \ref{thm:generalpreciseGaussian}, noting that the error parameter $\eta$ in the theorem is at most of order
\[\frac{(1+x)^3v_*}{\sqrt{k_*}} + (1+x)^2v\sqrt{k_*} + \frac{(1+x)^2k_*}{n}\leq (L''\,\kappa_{2,p})\frac{k_*^{\frac{7p-8}{4p-4}}}{n^{\frac{p-2}{4p-4}}} + L''\,\frac{k_*^2}{n}\leq C\gamma,\]
where $L''>0$ is another universal constant, and we adjust $C>0$ if needed to guarantee that the desired result holds.\end{proof}



\section{Moment-based bounds under contamination}\label{sec:proof:minimaxcontaminated}

In this section, we prove Theorem \ref{thm:minimaxcontaminated}: that is, we show that the trimmed mean can achieve minimax-optimal rates under contamination.

\begin{proof}[Proof of Theorem \ref{thm:minimaxcontaminated}]

To keep the notation similar to previous proofs, we take $x^2/2:=\log(4/\alpha)$ and obtain the following bound in terms of $x$.

\begin{equation}\label{eq:goalminimaxrestated}{\bf Goal:}\;\; \Pr{\left|\overline{X^\epsilon}_{n,k}-\mu\right|>C(d)\,\left(\nu_p\,\epsilon^{\frac{p-1}{p}} + \nu_q\,\frac{x^{\frac{2(p-1)}{p}}}{n^{\frac{p-1}{p}}}\right)\}\cap {\sf Good}}\leq 4e^{-\frac{x^2}{2}}.\end{equation}

We work under the framework of Proposition \ref{prop:richer}. That is, we assume that $X_i=F^{-1}(U_i)$ and $X_{(i)}=F^{-1}(U_{(i)})$ for each $i\in[n]$, where $\{U_i\}_{i=1}^n$ are i.i.d.\ uniform over $(0,1)$. 

A crucial step of the proof will be to use Proposition \ref{prop:effectcontamination} to relate $\overline{X^\epsilon}_{n,k}$ to an asymmetrically trimmed mean over the contaminated sample. Specifically, the proposition implies:
\[\overline{X}_{n,\lceil x^2\rceil,\lfloor 2\epsilon n\rfloor + \lceil x^2\rceil}\leq \overline{X^\epsilon}_{n,\lfloor \epsilon n\rfloor + \lceil x^2\rceil}\leq \overline{X}_{n,\lfloor 2\epsilon n\rfloor + \lceil x^2\rceil,\lceil x^2\rceil}.\]

It follows that 
\begin{equation}\label{eq:minimaxtwobounds}|\overline{X^\epsilon}_{n,\lfloor \epsilon n\rfloor + \lceil x^2\rceil}-\mu|\leq \max\{\overline{X}_{n,\lfloor 2\epsilon n\rfloor + \lceil x^2\rceil,\lceil x^2\rceil}-\mu,\mu-\overline{X}_{n,\lceil x^2\rceil,\lfloor 2\epsilon n\rfloor + \lceil x^2\rceil}\}.\end{equation}
For the remainder of the proof, we will focus on showing that there exists an event {\sf Good} with $\Pr{{\sf Good}^c}\leq 2e^{-x^2}$ such that
\begin{equation}\label{eq:goalminimax}\Pr{\left\{\overline{X}_{n,\lfloor 2\epsilon n\rfloor + \lceil x^2\rceil,\lceil x^2\rceil}-\mu>C(d)\,\left(\nu_p\,\epsilon^{\frac{p-1}{p}} + \nu_q\,\frac{x^{\frac{2(p-1)}{p}}}{n^{\frac{p-1}{p}}}\right)\right\}\cap {\sf Good}}\leq e^{-\frac{x^2}{2}}.\end{equation}
The same proof (with trivial modifications) shows that 
\[\Pr{\left\{\overline{X}_{n,\lceil x^2\rceil,\lfloor 2\epsilon n\rfloor + \lceil x^2\rceil}-\mu<-C(d)\left(\nu_p\,\epsilon^{\frac{p-1}{p}} + \nu_q\,\frac{x^{\frac{2(p-1)}{p}}}{n^{\frac{p-1}{p}}}\right)\right\}\cap {\sf Good}}\leq e^{-\frac{x^2}{2}}\]
and these two bounds, combined with the bound on $\Pr{{\sf Good}^c}$, imply the Theorem.

The proof of  (\ref{eq:goalminimax}) follows the general outline of Theorem \ref{thm:generalfinitevariance}. Setting $k_1=\lfloor 2\epsilon n\rfloor + \lceil x^2\rceil$, $k_2=\lceil x^2\rceil$, we apply Bernstein's inequality conditionally to $\overline{X}_{n,k_1,k_2}$ to obtain:
\begin{equation}\label{eq:bernsteinminimax}\Pr{\left.\overline{X}_{n,k_1,k_2}-\mu^{(U_{(k_1)},U_{(n-k_2+1)})}>\frac{x\sigma^{(U_{(k_1)},U_{(n-k_2+1)})}}{\sqrt{n-k_1-k_2}} + \frac{x^2\Delta_{n,k_1,k_2}}{12(n-k_1-k_2)}\right|(U_{(k_1)},U_{(n-k_2+1)})}\leq e^{-\frac{x^2}{2}}.\end{equation} 


Consider the event ${\sf Good}$ where all of the following inequalities holds:
\begin{eqnarray}\label{eq:conditionGoodXi}\Xi &\leq & \xi:=\frac{(\sqrt{k_1+k_2-1} + x)^2}{n},\\ \label{eq:conditionGoodDelta}
\Delta_{n,k_1,k_2}&\leq & 2\left(\frac{e^2n}{x^2}\right)^{\frac{1}{q}}\nu_q.\end{eqnarray} 

We now bound the quantities involved in (\ref{eq:bernsteinminimax}) under the event {\sf Good}. For that purpose, it is convenient to recall our assumption that 
\[(\sqrt{2\lfloor \epsilon n\rfloor +\lceil x^2\rceil} + x)^2=(\sqrt{k_1+k_2-1} + x)^2\leq dn\mbox{ with }d<1.\]
In what follows, we will  allow ourselves to write $C(d)$ for a positive constant depending only on $d$, whose exact value may change from line to line. For instance, this means that 
\[\frac{\xi^{\frac{p-1}{p}}}{1-\xi}\leq C(d)\,\left(\epsilon + \frac{x^2}{n}\right)^{\frac{p-1}{p}},\]
a fact that we will readily use. 

To continue, notice that, under {\sf Good}, by Proposition \ref{prop:controltrimmedmeanpop},
\[|\mu-\mu^{(U_{(k_1)},U_{(n-k_2+1)})}|\leq \frac{\nu_p\,\xi^{\frac{p-1}{p}}}{1-\xi}\leq C(d)\,\nu_p\,\epsilon^{\frac{p-1}{p}},\]
if $x\leq \epsilon n$; otherwise, we can similarly obtain a bound with $q$ replacing $p$ and  $x^2/n$ replacing $\epsilon$. In either case,
\[|\mu-\mu^{(U_{(k_1)},U_{(n-k_2+1)})}|\leq C(d)\,\left(\nu_p\,\epsilon^{\frac{p-1}{p}} + \nu_q\,\frac{x^{\frac{2(p-1)}{p}}}{n^{\frac{p-1}{p}}}\right).\]
By (\ref{eq:conditionGoodDelta}), we also have
\[\frac{x^2\Delta_{n,k_1,k_2}}{12\,(n-k_1-k_2)}\leq  C(d)\,\frac{x^{\frac{2(q-1)}{q}}}{n^{\frac{q-1}{q}}}\,\nu_q.\]
Additionally, Proposition \ref{lem:varboundmoments} gives
\begin{eqnarray*} \frac{x\sigma^{(U_{(k_1)},U_{(n-k_2+1)})}}{\sqrt{n-k_1-k_2}}&\leq & C(d)\,\frac{x\,\nu_q^{\frac{q}{2}}\Delta_{n,k_1,k_2}^{1-\frac{q}{2}}}{\sqrt{n}}\\ & = &C(d)  \left(\frac{x^2\Delta_{n,k_1,k_2}}{n}\right)^{1-\frac{q}{2}}\,\left(\nu_q\,\frac{x^{\frac{2(q-1)}{q}}}{n^{\frac{q-1}{q}}}\right)^{\frac{q}{2}}\\ \mbox{(use (\ref{eq:conditionGoodDelta}))} & \leq & C(d)\,\frac{x^{\frac{2(q-1)}{q}}}{n^{\frac{q-1}{q}}}\,\nu_q.\end{eqnarray*}
The upshot of these calculations is that
\[\Pr{\left\{|\overline{X}_{n,k_1,k_2}-\mu|>C(d)\,\left(\frac{x^{\frac{2(q-1)}{q}}}{n^{\frac{q-1}{q}}}\,\nu_q + \nu_p\,\epsilon^{\frac{p-1}{p}}\right)\right\}\cap {\sf Good}}\leq 2e^{-\frac{x^2}{2}},\]
as claimed. 

To finish, we need to show that $\Pr{{\sf Good}^c}\leq 2e^{-x^2/2}$. Since ${\sf Good}^c$ is the event where either (\ref{eq:conditionGoodXi}) of (\ref{eq:conditionGoodDelta}) do 
not hold, it suffices to bound the corresponding probabilities individually by $e^{-x^2/2}$. Indeed, this can be done via Propositions \ref{prop:xi} (with $t=x^2/2$) and \ref{prop:Deltasize} (with $k_{\min} = \lceil x^2/2\rceil\geq x^2/2$), respectively.\end{proof}

\section{Some illustrative experiments}\label{sec:experiments}

This section briefly compares the behavior of three estimators: the Catoni estimator from \cite{catoni2012challenging} with $\beta=1$, the sample mean and the trimmed mean with fixed $k=6$. 

Both the Catoni and the trimmed mean estimator enjoy added degrees of robustness relative to the sample mean, so we expect them to perform significantly better as the data comes from more challenging distribution. Below, we consider a $t$-distribution with degrees of freedom $\nu=1, 1.5, 2$ and $2.5$ --- the larger the degree of freedom, the less heavy-tailed the distribution becomes. Due to the added computational burden incurred by the Catoni estimator, we construct each estimator over $n=1000$ observations, and repeat the procedure $r=100$ times to obtain the violin plot. 

\begin{figure}[htpb]
  \centering
  \includegraphics[width=\textwidth]{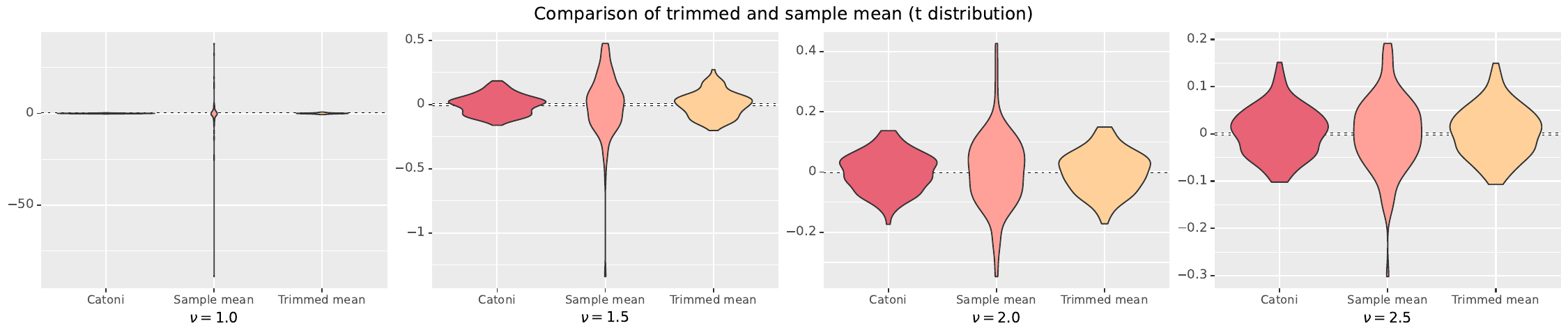}
  \caption{Violin plot for the three estimators under $t$ distributions with different parameters.
  \label{fig:violin-t}}
\end{figure}

Figure \ref{fig:violin-t} displays the observed histogram of each estimator through a violin graph, where a higher concentration around $0$ is better. It is clear that, for lower values of the degree of freedom, the sample mean is highly influenced by large but rare sample points; both the trimmed mean and the Catoni estimator are immune. As the degree of freedom increases, the three estimators become more similar, and more concentrated around the true mean of zero. We also note that the Catoni estimator seems to perform better than the trimmed mean in this case, although it is computationally much more expensive and does not have associated confidence intervals, such as the ones we propose for the trimmed means in this paper.

\appendix

\section{Some auxiliary technical results}

\subsection{Concentration of order statistics of uniforms}
\begin{lemma}[Upper tail concentration of order statistics]\label{lem:orderstats} Let $U_{(1)}\leq U_{(2)}\leq \dots\leq U_{(n)}$ be the order statistics of an i.i.d.\ ${\rm Uniform}[0,1]$ random sample. Then for all $k\in [n]$ and $t>0$:\end{lemma}
\begin{eqnarray}\label{eq:lower}\Pr{U_{(k)}>\frac{(\sqrt{k-1}+\sqrt{t})^2}{n}} = \Pr{1-U_{(n-k+1)}>\frac{(\sqrt{k-1}+\sqrt{t})^2}{n}} &\leq &   e^{-t},\\ 
\label{eq:upper}\Pr{U_{(k)}<\frac{(\sqrt{k}-\sqrt{t})^2}{n}} = \Pr{1-U_{(n-k+1)}<\frac{(\sqrt{k}-\sqrt{t})^2}{n}} &\leq &   e^{-2t}.\end{eqnarray}

\begin{proof}The equality of the two probabilities in each line follows from the symmetry of the uniform distribution under the transformation ``$u\mapsto 1-u$."

We will use two bounds for the binomial distribution proven in \cite[Theorems 3 and 4]{Okamoto1958} (see also \cite[Exercise 2.13]{boucheron2013concentration}): for all $c>0$, 
\begin{eqnarray*}\Pr{\sqrt{\frac{{\rm Binomial}(n,\lambda)}{n}}<\sqrt{\lambda}  - c} &\leq & e^{-c^2n},\\
\Pr{\sqrt{\frac{{\rm Binomial}(n,\lambda)}{n}}>\sqrt{\lambda} + c} &\leq & e^{-2c^2n}.\end{eqnarray*}

Now, for any $\lambda\in(0,1)$, $U_{(k)}>\lambda$ if and only if the number of $i\in [n]$ with $U_i\leq \lambda$ is less than $k$. This gives:
\[\forall \lambda\in (0,1)\,:\, \Pr{U_{(k)}>\lambda} = \Pr{\sum_{i=1}^n \mathbb{I}\{U_i\leq \lambda\}\leq k-1} = \Pr{{\rm Binomial}(n,\lambda)\leq k-1}.\]
For $\lambda\geq (k-1)/n$,
\[\Pr{{\rm Binomial}(n,\lambda)\leq k-1} = \Pr{\sqrt{\frac{{\rm Binomial}(n,\lambda)}{n}}<\sqrt{\lambda} - (\sqrt{\lambda} - \sqrt{\frac{k-1}{n}})}\leq e^{-(\sqrt{\lambda n} - \sqrt{k-1})^2}.\]
Taking:
\[\lambda=\frac{(\sqrt{k-1} + \sqrt{t})^2}{n}\]
as in the statement of the Lemma gives us (\ref{eq:lower}). For (\ref{eq:upper}), we note that $U_{(n-k+1)}>1-\lambda$ if there are at least $k$ points $U_i\in [1-\lambda,1]$. Using \cite[Theorem 3]{Okamoto1958}:
\[\Pr{\sum_{i=1}^n \mathbb{I}\{U_i>1- \lambda\}\geq k} = \Pr{\sqrt{\frac{{\rm Binomial}(n,\lambda)}{n}}>\sqrt{\lambda} + (\sqrt{\frac{k}{n}} - \sqrt{\lambda})}\leq e^{-2(\sqrt{k} - \sqrt{\lambda n})^2}.\]
The choice of $\lambda = (\sqrt{k}-\sqrt{t})^2/n$ gives us (\ref{eq:upper}).\end{proof}

\subsection{Perturbation bounds for the tail of the Gaussian}

The next result is a perturbation bound for the tails of the Gaussian distribution. In what follows, $\Phi$ is the standard Gaussian c.d.f.:
\[\Phi(x):=\int_{-\infty}^x\,\frac{e^{-\frac{t^2}{2}}}{\sqrt{2\pi}}\,dt\,\,\,(x\in\R).\]

\begin{proposition}\label{prop:tailgaussian} Let $x\geq 0$ and $h\in\R$ satisfy $|h|\leq (1/3\max\{x,1\})$. Then
\[\exp\left(-3|h|\max\{x,1\}\right)\leq \frac{1-\Phi(x+h)}{1-\Phi(x)}\leq \exp\left(3|h|\max\{x,1\}\right).\]
\end{proposition}
\begin{proof}We split the proof into three cases. \\

\noindent\underline{Case 1: $x<1$.} In this case $1-\Phi(x)\geq 1-\Phi(1)\geq 1-e^{-1/2}\geq 1/3$. Since the Gaussian c.d.f. is $L$-Lipschitz, with $L=(2\pi)^{-1/2}\leq 2$,
\[|1-\Phi(x+h) -(1-\Phi(x))|\leq \frac{|h|}{2},\]
and 
\[1 - \frac{3|h|}{2}\leq \frac{1-\Phi(x+h)}{1-\Phi(x)}\leq 1+\frac{3|h|}{2}.\]
We obtain the desired bound by noticing that $1+(3|h|/2)\leq e^{3|h|/2}$ and
\[e^{-3|h|}\leq 1-3|h| + \frac{(3|h|)^2}{2}\leq 1-\frac{3|h|}{2}\mbox{ for }0\leq |h|\leq 1/3.\]

\noindent\underline{Case 2: $x>1$ and $0\leq h\leq 1/x$.} Although the theorem only requires considering $x>1$ and $|h|\leq 1/3x$, it will turn out to be convenient to consider this wider range of $x,h$ in what follows. 

Lemma 16 in \cite{addarioberry2015} implies that, for any $x,h\geq 0$:
\[\frac{1-\Phi(x+h)}{1-\Phi(x)} = \exp\left(-hx-\frac{h^2}{2} -\eta\right)\]
for some $\eta\in (0,h/x)$. Since $x>1$ and $h\leq 1/3x$,
\[hx+\frac{h^2}{2}+\eta\leq hx + \frac{7h}{6x}\leq = \frac{13x}{6}.\]
We conclude that
\[e^{-13hx/6}\leq \frac{1-\Phi(x+h)}{1-\Phi(x)} \leq 1,\]
which is better than what we asked for.\\ 

\noindent\underline{Case 3: $x>1$ and $-(1/3x)\leq h\leq 0$.}
The idea will be to reapply the calculations of Case 2 with $x'=x-h$ replacing $x$ and $h'=|h|$ replacing $h$. To do this, we notice that $x'\geq x>1$ and 
\[x' =  x-h\leq x + \frac{1}{3x}\leq \frac{4x}{3}\leq 3x,\]
so $h'=|h|\leq 1/3x\leq 1/x'$. We deduce from Case 2 that
\[e^{-13h'x'/6}\leq \frac{1-\Phi(x'+h')}{1-\Phi(x')} \leq 1,\]
or equivalently
\[1\leq \frac{1-\Phi(x+h)}{1-\Phi(x)}\leq e^{13|h|(x+|h|)/6}\]
Since $|h|\leq 1/3x\leq x/3$, the exponent in the RHS is at most \[\frac{13|h|(x+|h|)}{6} \leq \frac{52\,|h|x}{18}\leq 3|h|x. \]\end{proof}

\subsection{Concentration of the empirical variance.}\label{sub:concentrationvariance}

The idea of the proofs of precise Gaussian approximation results is to combine self-normalized Central Limit Theorem with a concentration inequality for the empirical variance-like appearing in this result. In what follows, we present precisely this second inequality. 

\begin{theorem}\label{thm:bernsteinvariance}Consider i.i.d.\ random variables 
\[\{Z_i\}_{i=1}^n\mbox{ with $\Ex{Z_1}=\mu_Z,\,\Var{Z_1}\leq \sigma^2_Z$ and $|Z_1-\mu|\leq \Delta_Z$ almost surely.}\]
Define
\[V^2_n:=\frac{1}{n}\sum_{i=1}^n(Z_i-\mu_Z)^2\]
and
\[\Sigma^2_n:=\frac{1}{n}\sum_{i=1}^n(Z_i-\overline{Z}_n)^2,\]
where $\overline{Z}_n$ is the average of the $Z_i$ (cf. Theorem \ref{thm:bernstein}. Then: 
\[\Pr{\left|V_n -\sigma_Z\right|>\frac{z\Delta_Z}{\sqrt{n}}+ \frac{z^2\,\Delta^2_Z}{12n\,\sigma_Z}}\leq 2\exp\left(\frac{-z^2}{2}\right)\]
and 
\[\Pr{\left|\Sigma_n -\sigma_Z\right|>\frac{2z\Delta_Z}{\sqrt{n}}+ \frac{2z^2\,\Delta^2_Z}{12n\,\sigma_Z}}\leq 4\exp\left(\frac{-z^2}{2}\right).\]
\end{theorem}

\begin{proof}$V_n^2 - \sigma_Z^2$ is an average of i.i.d.\ random variables $W_i:=(Z_i-\mu_Z)^2-\sigma_Z^2$ that have mean $0$, are bounded above by $\Delta^2$ and have variances bounded by
\[\Var{W_i}\leq \Ex{(Z_i-\mu_Z)^4}\leq \Delta_Z^2\sigma_Z^2.\]
Applying Bernstein's inequality to $\pm V_n^2$ gives:
\[\Pr{|V_n^2 - \sigma_Z^2|>\frac{z\sigma_Z\,\Delta_Z}{\sqrt{n}} + \frac{z^2\,\Delta^2_Z}{12n}}\leq 2\exp\left(\frac{-z^2}{2}\right).\]
This finishes the proof of the first statement in the Theorem because 
\[|V_n - \sigma_Z| = \frac{|V_n^2-\sigma_Z^2|}{V_n+\sigma_Z}\leq \frac{|V_n^2-\sigma_Z^2|}{\sigma_Z}.\]
For the second statement, we simply notice that
\[\Sigma_n^2 + (\overline{Z}_n-\mu_Z)^2 = V^2_n,\]
so
\[|\Sigma_n - \sigma_Z|\leq |V_n - \sigma_Z|+|\overline{Z}_n-\mu_Z|\]
and that, under our assumptions
\[\Pr{|\overline{Z}_n-\mu_Z|>\frac{z\Delta_Z}{\sqrt{n}}+ \frac{z^2\,\Delta^2_Z}{12n\,\sigma_Z}}\leq 2e^{-\frac{x^2}{2}}\]
by Bernstein's inequality (Theorem \ref{thm:bernstein}).\end{proof}

\subsection{The trimmed mean is asymptotically efficient for fixed $k$} \label{sub:asymptoticefficiency}

Assume $\{X_i\}_{i=1}^{\infty}$ is an i.i.d.\ random sample with finite mean $\mu$ and variance $\sigma>0$. We sketch here a proof of the fact that the trimmed mean $\overline{X}_{n,k}$ is asymptotically Gaussian when $n\to +\infty$ and $k$ remains fixed. 

To start, recall that the (suitably normalized) sample mean converges to a Gaussian random variable:
\[\forall x\in\R\,:\, \Pr{\frac{\sqrt{n}}{\sigma}(\overline{X}_n-\mu)\leq x} \to \Phi(x).\]
Now
\[\overline{X}_{n,k}-\mu = \frac{n}{n-2k}(\overline{X}_n-\mu) - \frac{1}{n-2k}\left(\sum_{i=1}^{k-1}\,(X_{(i)}-\mu) + \sum_{j=n-k}^{n}\,(X_{(i)}-\mu)\right),\]
so
\[\left|\frac{\sqrt{n}}{\sigma}(\overline{X}_{n,k}-\mu)-\frac{\sqrt{n}}{\sigma}(\overline{X}_n-\mu)\right|\leq \frac{k}{(n-2k)}\left|\frac{\sqrt{n}}{\sigma}(\overline{X}_n-\mu)\right| + \frac{2k\sqrt{n}}{\sigma(n-2k)}\max_{1\leq i\leq n}|X_i-\mu|.\]
The fact that $\sigma<+\infty$ implies that as $n\to +\infty$, 
$\frac{\max_{1\leq i\leq n}|X_i-\mu|}{\sqrt{n}}\to 0\mbox{ almost surely}$. Therefore, the RHS of the preceding display goes to $0$ in probability. We conclude that
\[\forall x\in\R\,:\, \Pr{\frac{\sqrt{n}}{\sigma}(\overline{X}_{n,k}-\mu)\leq x} \to \Phi(x)\]
as well.

\subsection{Minimax lower bounds under moment conditions}\label{sub:minimaxmoments} We recall some minimax lower bounds for estimating the mean $\mu$ under moment conditions. The first one concerns the limits of statistical estimation from i.i.d.\ data under finite moment conditions. 

\begin{proposition}[\cite{dllo2016}, Theorem 3.1] There exists a universal constant $c>0$ such that the following holds. For any $\alpha\in (0,1/2)$, $q\in [1,2]$ and $\nu_q>0$, suppose $r>0$ and $E:\R^n\to \R$ is a measurable function such that
\[\Pr{|E(X_1,\dots,X_n) - \mu|\leq r}\geq 1-\alpha\]
for any $X_1,\dots,X_n$ i.i.d.\ random variables with mean $\mu=\Ex{X_1}$, a cumulative distribution function $F(t):=\Pr{X_1\leq t}$ ($t\in \R$), and with $\Ex{|X_1-\mu|^p}^{1/p}= \nu_q$. Then
\[r\geq c\,\nu_q\,\left(\frac{\log(1/\alpha)}{n}\right)^{\frac{q-1}{q}}.\]
\end{proposition}

In fact, the $p=2$ case of this Proposition holds even if $P$ is restricted to be Gaussian \cite[Proposition 5]{catoni2012challenging}. Therefore, assumptions about moments of order $p>2$ do not improve the random fluctuations term in mean estimation. 

The next result considers the case of contaminated data. In this case, higher moments do matter. The next result is essentially the same as \cite[Lemma 5.4]{minsker2018}; only the contamination model is slightly different. 

\begin{proposition}\label{prop:minsker}There exists a universal constant $c>0$ such that the following holds. Let $p\in [1,+\infty)$, $M_p>0$ and $\epsilon\in (6/n,1/2)$. Suppose $r>0$ and $E:\R^n\to \R$ is a measurable function such that
\[\Pr{|E(X^\epsilon_1,\dots,X^\epsilon_n) - \mu|\leq r}\geq \frac{2}{3}\]
for any $\epsilon$-contaminated random sample $X^\epsilon_1,\dots,X^\epsilon_n$ as defined in Theorem \ref{thm:minimaxcontaminated}, where the clean sample satisfies $\Ex{|X_1-\mu|^p}^{1/p}\leq M_p$. Then
\[r\geq c\,M_p\,\epsilon^{\frac{p-1}{p}}.\]
\end{proposition}

\begin{proof}It suffices to consider the case $M_p=1$. We adapt a strategy due to Minsker \cite[Lemma 5.4]{minsker2018} to our setting. 

The proof of \cite[Lemma 5.4]{minsker2018} shows the following. Given $p>1$, there exist distributions $P_1,P_2$ with centered $p$-th moment $\leq 1$ whose means $\mu_1$, $\mu_2$ satisfy
\[|\mu_1 - \mu_2| \geq 2c\,\epsilon^{1-1/p}\]
which moreover satisfy 
\[\left(1 - \frac{\epsilon}{2}\right)\,P_i + \epsilon Q_i = \hat{P}\,\,(i=1,2)\]
for certain distributions $Q_1,Q_2,\widehat{P}$. Crucially, note that $\hat{P}$ is the same for $i=1,2$, and $c>0$ is universal.

Let us now note that an i.i.d.\ random sample $\hat{X}_1,\dots,\hat{X}_n$ from $\hat{P}$ can be obtained from an i.i.d.\ random sample $X_1,\dots,X_n$ from $P_1$ as follows. First choose i.i.d.\ Bernoulli random variables $B_1,\dots,B_n$ with parameter $\epsilon/2$ independently from $X_1,\dots,X_n$. Now let $\hat{X}_i$ be drawn from $Q_1$ independently from everything else whenever $B_i=1$, and set $\hat{X}_i=X_i$ otherwise. 

From this description, we see that, conditionally on $\sum_{i=1}^nB_i\leq \epsilon n$, $\hat{X}_1,\dots,\hat{X}_n$ is an $\epsilon$ contamination of $X_1,\dots,X_n$, and therefore
\[\Pr{\left.|E(\hat{X}_1,\dots,\hat{X}_n) - \mu_1|> r\right| \sum_{i=1}^nB_i\leq \epsilon n}\leq \frac{1}{3}\]
by our assumption on $r,E$. 
\[\Pr{|E(\hat{X}_1,\dots,\hat{X}_n) - \mu_1|>r}\leq \frac{1}{3} + \Pr{\sum_{i=1}^nB_i>\epsilon n}<\frac{1}{2}\]
where in the last step we have implicitly used Chebyshev's inequality and the assumption that that $1/\epsilon n<1/6$.

We may repeat the above reasoning with $P_2,Q_2$ replacing $P_1,Q_1$ and deduce that we also have
\[\Pr{|E(\hat{X}_1,\dots,\hat{X}_n) - \mu_2|>r}<\frac{1}{2}.\]
In particular, there is a positive probability that 
\[|E(\hat{X}_1,\dots,\hat{X}_n) - \mu_1|\leq r \mbox{ and }|E(\hat{X}_1,\dots,\hat{X}_n) - \mu_2|\leq r\mbox{ happen simultaneously.}\]
This can only be if $r\geq |\mu_1-\mu_2|/2 =c\, \epsilon^{(p-1)/p}$.\end{proof}

\bibliographystyle{apalike}

\end{document}